\documentclass[a4paper]{article}
\usepackage{graphicx} 
\usepackage{amsmath, amsthm, amsfonts, amssymb, amscd, bm, enumerate}
\usepackage{verbatim}

\usepackage{url}

\newtheorem{theorem}{Theorem}
\newtheorem{corollary}{Corollary}
\newtheorem{lemma}{Lemma}
\newtheorem{definition}{Definition}
\newtheorem{assumption}{Assumption}

\theoremstyle{remark}
\newtheorem{remark}{Remark}
\DeclareMathOperator{\sign}{sign}
\def\B{\mathcal{B}}
\def\F{\mathcal{F}}
\def\E{\mathcal{E}}

\def\X{\mathcal{X}}
\def\U{\mathcal{U}}

\def\bR{\mathbb{R}}
\def\bN{\mathbb{N}}

\def\bone{\mathbf{1}}
\def\bP{\mathbb{P}}
\def\bQ{\mathbb{Q}}

\def\diag{\mathrm{diag}}
\mathchardef\mhyphen="2D
\def\tloc{t\mhyphen\mathrm{loc}}

\begin{document}
\title{Ergodic BSDEs driven by Markov Chains}
\author{Samuel N. Cohen\\ University of Oxford \\ \\ Ying Hu\\ Universit\'e de Rennes I}

\date{\today}

\maketitle

\begin{abstract}
We consider ergodic backward stochastic differential equations, in a setting where noise is generated by a countable state uniformly ergodic Markov chain. We show that for Lipschitz drivers such that a comparison theorem holds, these equations admit unique solutions. To obtain this result, we show by coupling and splitting techniques that uniform ergodicity estimates of Markov chains are robust to perturbations of the rate matrix, and that these perturbations correspond in a natural way to EBSDEs. We then consider applications of this theory to Markov decision problems with a risk-averse average reward criterion.

\end{abstract}

\section{Introduction}

Much work has gone into understanding optimal control with an average cost criterion, over an infinite horizon (for example, see the review paper Arapostathis et al. \cite{Arapostathis1993}, or more recent work by Guo and Hern\'andez-Lerma \cite{Guo2001}, and references therein). Provided an underlying controlled Markov process, this criterion provides a useful method for understanding those payoffs which value the future as much as the present, and hence are insensitive to short-term effects. Much of this work is constrained to look only at costs which depend on the current state of the process, and at the (linear) expectation of future costs; this is, in this sense, a linear theory (to be precise, the corresponding Hamiltonian operators are infima taken over linear functions in the costate variable, see Section \ref{sec:control}). Therefore, these methods are unable to deal adequately with risk-averse optimization, which requires a \emph{nonlinear} assessment of future costs.

 Given the known connection between BSDEs and the theory of `nonlinear expectations', as defined by Peng \cite{Peng1997} (see Coquet et al.\cite{Coquet2002} and Cohen \cite{Cohen2011a} for the precise details of this connection), it is not unreasonable to expect that `ergodic' BSDEs would provide a natural framework for understanding these nonlinear settings. 

In Fuhrman, Hu and Tessitore \cite{Fuhrman2009} (see also Richou \cite{Richou2009}, Debussche, Hu and Tessitore \cite{Debussche2011}), a class of value functions are considered which depend on the average cost, not only through the current state, but on the controlled stochastic dynamics, and can do so in a nonlinear way. These value functions are given by Ergodic Backward Stochastic Differential Equations (EBSDEs), a generalisation of the Backward Stochastic Differential Equations developed by Pardoux and Peng \cite{Pardoux1990}. In \cite{Fuhrman2009}, the stochastic dynamics are given with reference to a general cylindrical Wiener process in a Hilbert space, and the `current state' is described by a geometrically ergodic solution to a forward stochastic differential equation.

In this paper, we consider the ergodic BSDEs when noise is generated by a continuous-time discrete-state Markov chain. The basic BSDEs of this type, for finitely many states, were considered by one of the authors in \cite{Cohen2008, Cohen2008b, Cohen2011b}. We shall show that, with a discounted criterion, the infinite-horizon version of these BSDEs admits Markovian and time-invariant solutions. Assuming the underlying chain is uniformly ergodic, we show that discounted BSDEs preserve uniform ergodicity, in an appropriate sense. From here, we show that the ergodic BSDEs admit unique solutions.

The paper is organized as follows. In the next section, we deal with the theory of discounted BSDEs. In Section \ref{sec:ergodic}, we show the robustness of ergodicity estimates of Markov chains to perturbations of the rate matrix. We give a novel, but natural, partial ordering of rate matrices, and show that, given the perturbation is not too large as determined by this ordering, any ergodicity estimates are transferable. This result is itself of independent interest in the study of ergodic properties of Markov chains. Section \ref{sec:EBSDEs} proves the existence and uniqueness of bounded Markovian solutions to EBSDEs. Finally, in Section \ref{sec:example}, a numerical example and some applications of these equations in optimal ergodic control are discussed. We conclude with some thoughts on future extensions in Section \ref{sec:conclusion}.

\subsection{Introducing BSDEs and EBSDEs on Markov Chains}

Consider a continuous-time countable-state process $X$ in a probability space $(\Omega, \F, \bP)$, where $X$ will be our fundamental Markov chain.  Without loss of generality, we shall represent $X$ as taking values from the standard basis vectors $e_i$ of $\bR^N$ (where $N\in\bN\cup\{\infty\}$ is the number of states, and $\bR^\infty$ denotes the space of infinite real sequences). We write $\X$ for this set of basis vectors. For notational simplicity, we will think of all vectors as column vectors, and denote by $z^*$ the transpose of $z$ (so that $z^* y$ is the Euclidean or $\ell_2$ inner product). An element $\omega\in\Omega$ can be thought of as describing a path of the chain $X$.

Now let $\{\F_t\}$ be the completion of the filtration generated by $X$, that is,
\[\F_t = \sigma(\{X_s\}_{s\leq t}) \vee \{B\in\F_\infty:\bP(B)=0\}.\]
 As $X$ is a right-continuous pure jump process which does not jump at time $0$, this filtration is right-continuous, and we assume $\F=\F_\infty=\bigvee_{t<\infty}\F_t$. We shall assume that $X$ is a Markov chain under $\bP$, in the filtration $\{\F_t\}$.  For basic theory of continuous-time countable-state Markov chains, see for example Rogers and Williams \cite[Vol. 1, p228ff]{Rogers2000}, for the approach taken here, see Elliott, Aggoun and Moore \cite[Part III]{Elliott1994a}).

Let $A$ denote the possibly infinite rate matrix\footnote{In our notation, as in \cite{Elliott1994a}, $A$ is the matrix with entries $A_{ij}$, where $A_{ij}$ is the rate of jumping from state $j$ to state $i$. Depending on the convention used, this is either the rate matrix or its transpose. In our notation $A^*$, the transpose of $A$, is the generator of the Markov chain.} of the chain $X$. Note that $(A_t)_{ij}\geq 0$ for $i\neq j$ and $\sum_i (A_t)_{ij} = 0$ for all $j$ (the columns of $A$ all sum to $0$). We assume, for simplicity, that the entries in $A_t$ are uniformly bounded, and so the chain is regular.

From the Doob--Meyer decomposition (see \cite[Appendix B]{Elliott1994a}), we write our chain in the following way
\begin{equation}\label{ref:DMrepMarkovChain}
X_t = X_0 + \int_{]0,t]} A_u X_{u-} du + M_t,
\end{equation}
where $M$ is a locally-finite-variation pure-jump martingale in $\bR^N$, and the chain starts in state $X_0\in\bR^N$. Our aim is to study EBSDEs, that is, infinite-horizon equations of the form
\begin{equation}\label{eq:EBSDE}
 Y_t = Y_T+\int_t^T[f(X_{u-}, Z_u)-\lambda] du - \int_t^T Z_u^* dM_u, \qquad 0\leq t\leq T< \infty,
\end{equation}
where $f:\X\times\bR^N \to \bR$ is a given function, $Y$ is a real-valued c\`adl\`ag stochastic process, $Z$ is a predictable process in $\bR^N$ such that
\[\int_0^t Z_u^*dM_u := \sum_i \int_0^t (Z_u)^i d(M)^i_u\]
is a martingale, square integrable up to finite times, (here $(\cdot)^i$ denotes the $i$th component of the vector), and $\lambda\in\bR$. The key points distinguishing these equations from `classical' BSDEs on Markov chains (as considered in \cite{Cohen2008b}) are that this must hold both for all $t$ and for all $T$, and that the value $\lambda$ is a part of the solution.

Our key method will be based on considering the limits of the following `discounted' BSDEs,
\begin{equation}\label{eq:DBSDE}
 Y_t^{\alpha} = Y_T^{\alpha} +\int_t^T[f(X_{u-}, Z_u^{\alpha} )-\alpha Y^{\alpha}_{u-}] du - \int_t^T (Z^{\alpha}_u)^*dM_u, \qquad 0\leq t\leq T< \infty,
\end{equation}
the existence of solutions to which we shall first establish.

\begin{remark}
 We note that the use of left limits for $X$ and $Y$ in the driver terms of (\ref{eq:EBSDE}), (\ref{eq:DBSDE}) initially seems unconventional, for those used to the theory of BSDEs in a Brownian setting. However, it is the natural approach when the driver term can itself jump (see \cite{Cohen2010}), it allows us to perform Girsanov transformations without constantly having to switch between the left and right limits of the processes, and as the integral is with respect to Lebesgue measure and our processes have at most countably many jumps, in this case the equation is unchanged whether the left limits are included or not.
\end{remark}

Of importance will be the following process and the associated spaces.
\begin{definition}
 Let
\[\psi_t^x:= \diag(A_t x) - A_t \diag(x) - \diag(x) A_t^*,
\]
for $x$ a basis vector of $\bR^N$. Write $\psi_t$ for the process $\psi_t^{X_{t-}}$. Then $\psi$ is a predictable process taking values in the symmetric, positive semidefinite matrices in $\bR^{N\times N}$, with the property that
\[E\Big[\Big(\int_{]0,t]} Z_u^*dM_u\Big)^2\Big] = \int_{]0,t]} E[Z^*_u \psi_u Z_u] du\]
for any $t$ and any predictable processes $Z$ of correct dimension (see \cite{Cohen2008}). For simplicity, we write
\[\|z\|^2_{M_t} := z^*\psi_t z,\]
and note that this is a stochastic seminorm.

We define the following spaces of processes.
\begin{itemize}
 \item $Y\in S^2$ if $E[\sup_t Y^2_t]<\infty$ and $Y$ is c\`adl\`ag,
 \item $Z\in H^2_M$ if $E\big[\int_{]0,\infty]} \|Z_t\|_{M_t}^2 dt\big] < \infty$ and $Z$ is predictable,
 \item $Z\in H^2_{M,\tloc}$ if $E\big[\int_{]0,t]} \|Z_s\|_{M_s}^2 ds\big] < \infty$ for all $t<\infty$ and $Z$ is predictable (note that this is not the usual space of processes locally in $H^2_M$, as this must hold for every deterministic $t$, rather than for a specific sequence of stopping times),
 \item $Z\sim_M Z'$ if $\|Z_t-Z'_t\|_{M_t}=0$ for almost all $t$.
\end{itemize}
\end{definition}

\section{Discounted BSDE}
We begin with the following result.

\begin{theorem}\label{thm:basicBSDEexist}
Let $T$ be a finite deterministic time, and $f:\Omega\times[0,T]\times\bR\times\bR^N \to\bR$ be a predictable function. If $f$ is uniformly Lipschitz in $y$ and $z$, that is, there exists a constant $c>0$ such that
\[|f(\omega, t, y, z)- f(\omega,t,y', z')| \leq c(|y-y'| + \|z-z'\|_{M_t}),\]
and
\[E\Big[ \int_{]0,T]} |f(\omega, t, 0, 0)|^2 dt \Big] < \infty\]
then for any $\xi\in L^2(\F_T)$, there exists a unique solution $(Y,Z) \in S^2 \times H^2_M$ to the BSDE
\[\xi=Y_t- \int_{]t,T]} f(\omega, u, Y_u, Z_u) du + \int_{]t,T]} Z_u^* dM_u.\]
\end{theorem}

\begin{proof}
For the finite state case, this result is given in \cite{Cohen2008}. For the infinite state case, we use the martingale representation result established in \cite{Cohen2008},  which naturally extends to general spaces, coupled with the existence result for BSDEs in general spaces established in \cite{Cohen2010}.

We note that, unlike in \cite{Cohen2008,Cohen2010}, we have not assumed that $\F_0$ is trivial, that is, that $X_0$ is deterministic. Hence $Y_0$ is also a random variable. This poses no problems for the theory of BSDEs, however it will be useful for us to note that, on the set $\{X_0=x\}$, we will obtain a deterministic value $Y^x_0$, as $\F_0$ is nothing but the completion of $\sigma(X_0)$.
\end{proof}

A key result in the analysis of BSDEs is the comparison theorem. In the case of BSDEs with Markov Chain noise, and in general for BSDEs with jumps, a further condition is required to ensure that the result holds. In \cite{Cohen2009, Cohen2011a}  a general condition under which the comparison theorem holds is presented, and in \cite{Cohen2008b} a condition specific to Markov chain BSDEs was also given.

\begin{definition}\label{defn:balanced}
 For a driver $f$, we say that
\begin{itemize}
 \item $f$ is \emph{balanced} if 
\[\frac{f(\omega, t, y, z)- f(\omega,t,y,z')}{\|z-z'\|^2_{M_t}}(z-z')^* \Delta M_t >-1,\]
\item $f$ is \emph{weakly balanced} if 
\[\bigg(\frac{f(\omega, t, y, z)- f(\omega,t,y,z')}{\|z-z'\|^2_{M_t}} \wedge 0\bigg) (z-z')^* \Delta M_t >-1,\]
\item  $f$ is \emph{strictly balanced} if
\[\frac{f(\omega, t, y, z)- f(\omega,t,y,z')}{\|z-z'\|^2_{M_t}}(z-z')^* \Delta M_t >-1+\gamma\]
for some $\gamma>0$,
\end{itemize}
where all inequalities must hold for any $y\in \bR$, any $z,z'\in\bR^N$, and up to indistinguishability.
\end{definition}
Clearly strictly balanced $\Rightarrow$ balanced $\Rightarrow$ weakly balanced. Our main attention is on those cases where $f$ is (strictly) balanced, however, the comparison theorem holds assuming only that $f$ is weakly balanced.

\begin{theorem}[Finite-time comparison theorem]
Let $(Y, Z)$ and $(Y', Z')$ be the solutions to two BSDEs with drivers $f$ and $f'$. Suppose $f$ is weakly balanced and $f(\omega, t, y, z)\geq f'(\omega, t, y, z)$ for all $(y,z)$, $dt\times d\mathbb{P}$-a.s.  Then $Y_T\geq Y'_T$ a.s. implies $Y_t\geq Y'_t$ a.s. up to indistinguishability.
\end{theorem}
\begin{proof}
 This is simply Theorem 3 of \cite{Cohen2011a}, where a trivial modification of the last line of the proof of Lemma 1 of \cite{Cohen2011a} is needed to exploit the definition of a weakly balanced driver.
\end{proof}

In fact, in the balanced case, the proof  of the comparison theorem is easy to deduce from the following lemma.
\begin{lemma}\label{lem:girsanovvalid}
 If $f$ is Lipschitz and balanced, then for any predictable processes $Z,Z'\in H^2_{M,\tloc}$, any process $Y\in S^2$, any $T<\infty$, the measure $\bQ^T$ defined by
\[\frac{d\bQ^T}{d\bP} = \E\Big(\int_{]0,\cdot\wedge T]} \frac{f(\omega, t, Y_{t-}, Z_t)- f(\omega,t,Y_{t-},Z'_t)}{\|Z_t-Z'_t\|^2_{M_t}} (Z_t-Z'_t)^* dM_t\Big),\]
 is a probability measure (where $\E$ denotes the Dol\'eans-Dade exponential), and
\[\tilde M_t=\int_{]0,t\wedge T]}\big(f(\omega, s, Y_{s-}, Z_s)- f(\omega,s,Y_{s-},Z'_s) \big)ds + \int_{]0,t\wedge T]}(Z_s-Z'_s)^*dM_s\]
is a $\bQ^T$-martingale.
\end{lemma}
\begin{proof}
First note that
\[L_t := \int_{]0,t\wedge T]} \frac{f(\omega, s, Y_{s-}, Z_s)- f(\omega,s,Y_{s-},Z'_s)}{\|Z_s-Z'_s\|^2_{M_t}} (Z_s-Z'_s)^* dM_s\]
is a local martingale with quadratic variation
\[\langle L\rangle_t = \int_{]0,t\wedge T]} \frac{|f(\omega, s, Y_{s-}, Z_s)- f(\omega,s,Y_{s-},Z'_s)|^2}{\|Z_s-Z'_s\|^4_{M_t}} \|Z_s-Z'_s\|^2_{M_t} dt \leq cT,\]
for $c$ a Lipschitz constant of $f$. Hence we know $\tilde M_t^T$ is a true (BMO-)martingale with all moments finite (see \cite[Lemma 2]{Cohen2011a}). Then, as $f$ is balanced, we see that $\Delta L_t>-1$, so $\bQ^T$ is a probability measure.

To show that the desired process is a local martingale is then an application of Girsanov's theorem. By H\"older's inequality we have, for any stopping time $\tau\leq T$, any $\epsilon<1$,
\[E_{\bQ^T}[\tilde M^{2-\epsilon}_\tau] \leq E_{\bP}[L_T^{2/\epsilon}]^{\epsilon/2}  E_{\bP}[\tilde M^2_\tau]^{1-\epsilon/2}\]
which is uniformly bounded, so $\tilde M$ is a true $\bQ^T$-martingale.
\end{proof}

The connection between these definitions of `balanced' drivers and the condition given  in \cite{Cohen2008b} is given by the following lemma, which is presented for completeness

\begin{lemma}\label{lem:balancedequiv}
 The following conditions are equivalent.
\begin{enumerate}[(i)]
 \item $f$ is weakly balanced.
 \item For any $z,z'\in \bR^N$, up to indistinguishability, on the set where
\[(e_i^*A_tX_{t-})[(z-z')^*(e_i-X_{t-})] \geq 0\]
for all $i$, we have
\[(f(\omega, t, y,z) - f(\omega, t, y, z'))\wedge 0 \geq -(z-z')^*A_tX_{t-}\qquad a.s.,\]
with equality only when $\|z-z'\|_{M_t}=0$.
\end{enumerate}
\end{lemma}
\begin{proof}
Clearly the conditions are trivial on the set $f(\omega, t, y, z) - f(\omega, t, y, z')=0$, so we exclude this from consideration.

\emph{$(ii)\Rightarrow (i)$.} First note that the condition in Definition \ref{defn:balanced} is equivalent to
\begin{equation}\label{eq:balancedsimple}
 [(f(\omega, t, y, z) - f(\omega, t, y, z'))\wedge 0] (z-z')^*(X_t-X_{t-}) > -(z-z')^*\psi_M(z-z')\qquad a.s.
\end{equation}

For fixed values of $t,y,z, X_{t-}$, suppose there is $i$ with $(e_i^*A_tX_{t-})[(z-z')^*(e_i-X_{t-})] < 0$. Then let
\[h=\bigg(\frac{f(\omega, t, y, z) - f(\omega, t, y, z')}{(z-z')^*(e_i-X_{t-})}\vee 0\bigg) (e_i-X_{t-}).\]
If there is no such $i$, then note that $(z-z')^*A_tX_{t-}< 0$ (as it is the compensator of a nondecreasing process), and let
\[h=\bigg(\frac{f(\omega, t, y, z) - f(\omega, t, y, z')}{(z-z')^*A_tX_{t-}}\vee 0\bigg) A_tX_{t-}\]
In either case, $h^*(z-z')=(f(\omega, t, y, z) - f(\omega, t, y, z'))\wedge 0$, $\bone^*h=0$ and $h^*e_i\geq 0$ for all $e_i\neq X_{t-}$. In the latter case, note also that $\bone^*(A_tX_{t-}-h)=0$ and $e_i^*(A_tX_{t-}-h)> 0$ for all $e_i\neq X_{t-}$.

By considering all possible jumps, we then have that (\ref{eq:balancedsimple}) simplifies to
\[(z-z')^*[h(e_j-X_{t-})^* + \psi_t] (z-z')>0\]
for all $e_j$ with $e_j^*A_tX_{t-}>0$. This is equivalent to
\[(z-z')^*[h(e_j-X_{t-})^*+ (e_j-X_{t-}) h^* + 2\psi_t] (z-z')>0.\]
As $z-z'$ is at most defined up to the addition of a constant, for each $e_j$, we can suppose without loss of generality that $(z-z')^*e_j=0$. Hence it is enough to show that the symmetric matrix $[2\psi_t-hX_{t-}^*-X_{t-} h^*]$ is positive definite. However, as this matrix is of the same form as $\psi_t$, this is straightforward.

\emph{$(i)\Rightarrow (ii)$.} We know that (\ref{eq:balancedsimple}) holds. Suppose that $(e_i^*A_tX_{t-})[(z-z')^*(e_i-X_{t-})] \geq 0$ for all $i$. Without loss of generality, we select a representation of $z-z'$ such that $(z-z')^*X_{t-}=0$ and $(z-z')^*e_j=0$ for all $j$ with $e_j^*A_tX_{t-}=0$. Note that this then implies $z-z'$ is componentwise nonnegative. Then (\ref{eq:balancedsimple}) reduces to
\[ (f(\omega, t, y, z) - f(\omega, t, y, z'))\wedge 0 > -\frac{(z-z')^*\psi_M(z-z')}{ (z-z')^*(X_t)} = -\frac{\sum_i a_i(z-z')^2_i}{ (z-z')^*(X_t)}\]
with the convention $0/0=\infty$, where $a_i = e_i^* A_t X_{t-}$. This inequality must hold almost surely, so it must hold in the case $X_t=e_j$, for $j$ maximizing $(z-z')^*e_j=(z-z')_j$. As $(z-z')_i/(z-z')_j<1$, we have
\[ \begin{split}(f(\omega, t, y, z) - f(\omega, t, y, z'))\wedge 0 &> -\frac{\sum_i a_i(z-z')^2_i}{ (z-z')_j}
\\&>-\sum_i a_i(z-z')_i = -(z-z')^* A_tX_{t-}\end{split}\]
as desired.
\end{proof}

\begin{corollary}\label{cor:strictbalanceimp}
 If $f$ is balanced, then we also have that for any $z,z'\in \bR^N$, up to indistinguishability, if
\[(e_i^*A_tX_{t-})[(z-z')^*(e_i-X_{t-})] \geq 0\]
for all $i$, then
\[f(\omega, t, y,z) - f(\omega, t, y, z') \geq -(z-z')^*A_tX_{t-}\qquad a.s.,\]
with equality only when $\|z-z'\|_{M_t}=0$.
\end{corollary}
\begin{proof}
 The proof of {$(i)\Rightarrow (ii)$} in Lemma \ref{lem:balancedequiv} is sufficient to prove this result, where the term `$\wedge 0$' is removed throughout.
\end{proof}

The following version of Tanaka's formula will be useful. We adopt the convention that $\sign(x)=x/|x|$ for $x\neq 0$ and $\sign(0)=0$.
\begin{lemma}
 Let $Y$ be a c\`adl\`ag process of finite variation. Then the dynamics of $|Y|$ are given by
\[d|Y|_t = \sign(Y_{t-}) dY_t + \Delta L^Y_t\]
where $\Delta L^Y_t$ is the `local time' jump process, with
\[\Delta L_t^Y = |Y_{t}|(1-\sign(Y_tY_{t-}))\geq 0.\]
\end{lemma}
\begin{proof}
 The dynamics of $|Y|$ are clear except when $Y_{s-}=0$ or when $Y$ jumps over zero, that is, when $Y_{t-}$ and $Y_t$ are of opposite sign.

If $Y_{s-}=0$, then either $|Y_s|>0$, in which case $\Delta|Y_s|=\Delta L_t^Y$, or $Y_s=0$, in which case there is a version of the derivative of $Y$ which is zero at $s$ (as the measure induced by $Y$ does not charge the point $s$). In either case, the dynamics are satisfied.

If $Y_{t-}$ and $Y_t$ are of opposite sign, we have
\[\begin{split}\Delta |Y|_t& = \sign(Y_{t-}) \Delta Y_t + (\Delta |Y|_t - \sign(Y_{t-}) \Delta Y_t)\\
   &= \sign(Y_{t-}) \Delta Y_t + (|Y_t|-|Y_{t-}| - \sign(Y_{t-}) (Y_t - Y_{t-}))\\
 &= \sign(Y_{t-}) \Delta Y_t + (|Y_t|- \sign(Y_{t-})Y_t)\\
&= \sign(Y_{t-}) \Delta Y_t + |Y_t|(1-  \sign(Y_{t-}Y_t))\\
  \end{split}
\]
and so in both cases the dynamics hold.
\end{proof}

We now seek to show that there exists a bounded solution to the infinite-horizon BSDE with discounting. The approach used to obtain this result is based on Briand and Hu \cite{Briand1998} and Royer \cite{Royer2004}. However, in our setting, the following result greatly simplifies the analysis.

\begin{lemma}\label{lem:Zbound}
Let $Y$ satisfy the dynamics
\[dY_t =  \beta_t dt + Z_t^* dM_t\]
for some arbitrary predictable process $\beta$. If $|Y|\leq k$ for some $k$, then $|e_i^*Z|\leq 2k$ for all $i$, for at least one representative in $H^2_{M, \tloc}$. That is, $Z$ is bounded componentwise by twice a bound on $Y$.
\end{lemma}
\begin{proof}
 First note that as the jumps of $M$ are totally inaccessible, if an inequality holds for every jump time, then it must hold almost everywhere on every set where jumps occur with positive probability, that is, up to a representative in $H^2_{M, \tloc}$.  Now note that at any jump,
\[Y_t -Y_{t-} = \Delta Y_t = Z_t^*\Delta M_t = Z_t^* (X_t-X_{t-}).\]
Therefore,
\[|Z_t^* (X_t-X_{t-})|\leq 2k \quad a.s.\]
We now take a representative $Z$ such that $Z^*X_{t-}\equiv 0$, which can be done as $Z$ is only ever defined up to addition of a constant. Therefore $|Z_t^* X_t|\leq 2k$ for every $X_t$ which can be reached with positive probability in a single jump. Taking $Z_t=0$ for all $X_t$ which cannot be reached in a single jump (which can be done up to equality $\sim_M$), we see that $|e_i^*Z|\leq 2k$.
\end{proof}

\begin{theorem}\label{thm:discountedBSDEexist}
Let $\alpha>0$ and $f:\Omega\times\bR^+\times \bR^N\to \bR$ be
\begin{itemize}
 \item uniformly Lipschitz (in its third component) with respect to the $\|\cdot\|_{M_t}$-norm $d\bP\times dt$-a.e.
\item balanced, in the sense of Definition \ref{defn:balanced} (omitting the $y$ variable), and
\item such that $|f(\omega, t, 0)|$ is uniformly bounded by $C\in\bR$.
\end{itemize}
Then there exists an adapted solution $(Y, Z)$, with $Y$ c\`adl\`ag and $Z\in H^2_{M,\tloc}$, to the equation
\begin{equation}\label{eq:discBSDE}
Y_T=Y_t - \int_{]t,T]} (-\alpha Y_{u-} + f(\omega, u, Z_u)) du + \int_{]t,T]} Z_u^*dM_u, \qquad 0\leq t\leq T<\infty
\end{equation}
satisfying $|Y_t|\leq C/\alpha$, and this solution is unique among bounded adapted solutions.

Furthermore, if $(Y^T, Z^T)$ denotes the (unique) adapted solution to
\begin{equation}\label{eq:ThorizonBSDE}
 0=Y^T_t - \int_{]t, T]} (-\alpha Y_{u-}^T + f(\omega, u, Z_u^T)) du + \int_{]t, T]} (Z_u^T)^*dM_u
\end{equation}
then $\lim_{T\to\infty} Y^T_t = Y_t$ a.s., uniformly on compact sets in $t$.
\end{theorem}

\begin{proof}
 \textbf{Uniqueness.} We first show that, if a bounded solution exists, it is unique. Suppose we have two bounded solutions $(Y, Z)$ and $(Y', Z')$ to (\ref{eq:discBSDE}). For simplicity, write $\delta Y= Y-Y'$ and $\delta Z=Z-Z'$.

For $T>0$, let $\bQ^T_1$ be the measure with density
\begin{equation}\label{eq:Q1}
 \frac{d\bQ^T_1}{d\bP} = \E\Big(\int_{]0,\cdot \wedge T]} \frac{(f(\omega, u, Z_u) - f(\omega, u,  Z'_u))}{\|\delta Z_u\|^2_{M_u}}(\delta Z_u)^*dM_u\Big),
\end{equation}
where $\E$ denotes the Dol\'eans-Dade exponential. As $f$ is balanced, we can see that $\bQ^T_1$ is a probability measure. By Lemma \ref{lem:girsanovvalid}, it follows that
\[\begin{split}&-\int_{]0,t]} \sign(\delta Y_{u-}) (f(\omega, u, Z_u) - f(\omega, u, Z_u')) du \\& \qquad+ \int_{]0,t]} \sign(\delta Y_{u-}) \delta Z_u^* dM_u +\sum_{u\leq t}\Delta L^{(\delta Y_u)}_u \end{split}\]
is a $\bQ^T_1$ submartingale on $[0,T]$. (Note that the inclusion of $\sign(\delta Y_{u-})$ simply exchanges $Z$ and $Z'$, and does not affect the quadratic variation.)

By Tanaka's formula and It\=o's formula, we have, for all $s\leq t\leq T$,
\[E_{\bQ^T_1}[e^{-\alpha t}|\delta Y_t| - e^{-\alpha s}|\delta Y_s|\,|\F_s] \geq 0,\]
hence,
\[|\delta Y_s| \leq e^{-\alpha (t-s)}E_{\bQ^T_1}[|\delta Y_t|\,|\F_s] \leq e^{-\alpha (t-s)} C,\] for $C$ a bound on $|\delta Y_t|$. This bound is independent of $T$, and collapses as $t\to\infty$. Hence $|\delta Y_s|=0$, from which we see $Y_s=Y'_s$ a.s. for every $s$, and hence $Y=Y'$ up to indistinguishability as $Y$ and $Y'$ are c\`adl\`ag.

\textbf{Existence.} We now show that a bounded solution exists. Let $(Y^T, Z^T)$ denote the solution to the time $T$-horizon BSDE, as defined in (\ref{eq:ThorizonBSDE}). 

First, we show that $Y^T$ is bounded. Again, we invoke Lemma \ref{lem:girsanovvalid} and let $\bQ^T_2$ denote the probability measure with density
\[\frac{d\bQ^T_2}{d\bP} = \E\Big(\int_{]0,\cdot \wedge T]} \frac{f(\omega, u, Z_u^T) - f(\omega, u, 0)}{\|Z_u\|^2_{M_u}}(Z_u^T)^* dM_u\Big).
\]
Applying Tanaka's formula and It\=o's formula to $e^{-\alpha t}|Y^T_t|$, we see that
\begin{equation}\label{eq:boundsoln}
 |Y^T_t| \leq e^{\alpha t}E_{\bQ^T_2}\Big[\int_t^T e^{-\alpha u}|g(\omega, u, 0)|du\Big|\F_t\Big]\leq C/\alpha
\end{equation}
for $C$ a bound on $|f(\omega, u, 0)|$. Hence $Y^T$ is uniformly bounded.

Second, we show that $Y^T$ forms a Cauchy sequence in $T$. For any $T'>T$, we use Lemma \ref{lem:girsanovvalid} to take the probability measure
\[\frac{d\bQ^{T, T'}_3}{d\bP} = \E\Big(\int_{]0,\cdot \wedge T]} \frac{f(\omega, u, Z_u^T) - f(\omega, u, Z_u^{T'})}{\|Z_u^T-Z_u^{T'}\|^2_{M_u}}(Z_u^T-Z_u^{T'})^* dM_u\Big).
\]
 Again applying Tanaka's formula, It\=o's formula and the inequality (\ref{eq:boundsoln}), for $t<T$ we have
\[|Y^T_t- Y^{T'}_t| \leq e^{-\alpha(T-t)} E_{\bQ^{T, T'}_3}[|Y^{T'}_T|\,|\F_t] \leq Ce^{-\alpha(T-t)}/\alpha\qquad a.s.\]
Hence we see that $Y^T_t$ is a Cauchy sequence in $T$, and so a limit exists, and we denote it $Y_t$. The desired convergence uniformly on compacts is clear, and the bound established in (\ref{eq:boundsoln}) also holds for $Y_t$.

Finally, as $Y^T_t$ is uniformly bounded, so is $Z^T_t$, by Lemma \ref{lem:Zbound}. As $Y^T_t$ converges a.s.~uniformly on compacts in $t$, its jumps converge a.s.~uniformly on compacts, however this implies that the $Z^T_t$ also converge a.s.~uniformly on compacts, at least up to equivalence in $\|\cdot\|_{M_t}$. Taking $Z$ as the limit of $Z^T$, we have our desired solution $(Y,Z)$.
\end{proof}

To finish this section, we finally state a result on the existence of `Markovian' solutions to these BSDEs, that is, when the BSDE solution $(Y,Z)$ can be written as a function of the underlying state process $X$.

\begin{theorem}\label{thm:markovianDisc}
Suppose our Markov chain $X$ is irreducible at every point of time. Consider either
\begin{itemize}
 \item the situation of Theorem \ref{thm:basicBSDEexist}, when the terminal value of the BSDE is of the form $Q=\phi(X_T)$ for some deterministic function $\phi:\bR^N\to\bR$, or
\item the situation of Theorem \ref{thm:discountedBSDEexist},
\end{itemize}
and suppose that the driver $f$ factors through $X_{t-}$, as a function of $\omega$, that is, $f$ can be written in the form $f(\omega, t, \cdots) = \tilde f(X_{t-}(\omega), t, \cdots)$ for some function $\tilde f$. Then there exists a function $v:[0,\infty[\,\times \X\to\bR$ such that $Y_t = v(t, X_t)$ and $e_i^*Z_t = v(t, e_i)$, and this function satisfies the coupled ODE system
\[\frac{d v(t, e_i)}{dt} = -\tilde f(e_i, t, v(t, e_i), v(t, \cdot)) - e_i^* A_t^* v(t, \cdot)\]
in the setting of Theorem \ref{thm:basicBSDEexist}, respectively 
\[\frac{d v(t, e_i)}{dt} = -\alpha v(t, e_i) -\tilde f(e_i, t, v(t, \cdot)) - e_i^* A_t^* v(t, \cdot),\]
in the setting of Theorem \ref{thm:discountedBSDEexist}, where $v(t, \cdot)$ refers to the vector in $\bR^N$ with entries $\{v(t, e_j)\}_{j=1}^N$.
\end{theorem}
\begin{proof}
 For $N$ finite, in the first case, this result is directly from \cite{Cohen2011b}. It is easy to verify that the same arguments as in \cite{Cohen2011b} will hold for $N$ infinite. For the situation of Theorem \ref{thm:discountedBSDEexist}, note that the finite-time approximations $Y^T_t$ constructed in the proof of that theorem are all examples of the first case, with $\phi\equiv 0$. Hence we can find functions $v^T(\cdot, \cdot)$ satisfying the desired statements, and the proven convergence allows us to take $T\to\infty$. The only difficulty with this is the stated dynamics on $v(t, e_i)$, which can easily be verified directly, as in \cite{Cohen2011b}.
\end{proof}

\begin{corollary}\label{cor:markoviansolutions}
 Suppose that $X$ is time-homogeneous (so $A=A_t$ is constant), and we are in the situation of Theorem \ref{thm:discountedBSDEexist}, where our driver factors through $X_{t-}$, and does not depend on time. Then $v$ does not depend on time, and we have the equation in $\bR^N$
\[\alpha v = -\tilde f(\cdot, v) -  A^*v\]
where $v=v(t, \cdot)$ is a vector in $\bR^N$, and $\tilde f(\cdot, v)$ refers to the vector with entries $\{\tilde f(e_i, v)\}_{i=1}^N$. If $\tilde f$ is balanced, using the natural modification of Definition \ref{defn:balanced} (and correspondingly Lemma \ref{lem:balancedequiv}) this equation admits a unique solution.
\end{corollary}

\begin{proof}
 Under the conditions of the corollary, as $X$ is a Markov chain, our infinite-horizon BSDE (\ref{eq:discBSDE}) does not vary in $t$ given the state $X_t$. It is then clear that the unique solution is a function purely of $X_t$. The dynamics are then the natural rewriting of those in Theorem \ref{thm:markovianDisc}, and the existence and uniqueness of the solution is the result of Theorem \ref{thm:discountedBSDEexist}.
\end{proof}

\begin{remark} 
 It is worth noting that this is a nontrivial algebraic statement, due to the nonlinearity of the function $\tilde f$. If we knew that $\tilde f(v)+A^*v$ was a strictly monotone function, in the sense that
\[
\langle \tilde f(v)-\tilde f(v')+A^*(v-v'), v-v'\rangle \leq -\epsilon\|v-v'\|^2
\]
 for some $\epsilon>0$ and all $v,v'$, then it would be possible to construct an existence result using standard techniques (see, for example \cite[p.565]{Zeidler1989}). Initially, this would appear to be true, at least for chains with finitely many states, using the fact that the nonlinear equation generates measures under which $X$ is a Markov chain (see the proof of Lemma \ref{lem:PerturbErgodic}), and the Perron--Frobenius theorem or Krein--Rutman theorem to bound eigenvalues below zero. However, such an argument depends on the diagonalizability of the derivative of $\tilde f(v)+A^*v$ which is a non-trivial assumption. For example, one might consider the situation on four states with $\tilde f=0$, and
\begin{equation}\label{eq:pathA}
 A^*= \left[\begin{array}{cccc} -3 &1&2&0\\1&-3&2&0\\ 0&2&-3 &1\\ 0&2&1&-3 \\ \end{array}\right]
\end{equation}
which is a very well behaved Markov chain generator, however is defective as a matrix and so it is easy to find vectors $v$ with $v'Av>0$. Hence this argument fails in general.
\end{remark}

\section{Uniformly ergodic chains}\label{sec:ergodic}
We aim to construct solutions to the ergodic BSDE (\ref{eq:EBSDE}), using the solutions to the discounted BSDEs (\ref{eq:discBSDE}). As our aim is to contemplate the long-run behaviour, we shall now make the following two assumptions, for the remainder of the paper. First, the Markov chain $X$ is time homogeneous (so $A_t=A$ for all $t$) and irreducible, and second, the BSDE drivers we consider will all factor through $X$, so our BSDE solutions will be functions of the current state of the Markov chain, and, in particular, Corollary \ref{cor:markoviansolutions} holds.

The key to using discounted BSDEs to approximate the ergodic BSDE is to study the ergodic behaviour of the Markov chain itself. We shall therefore make moderately restrictive assumptions on the Markov chain, sufficient to give explicit bounds on its convergence to its ergodic distribution, and then show that these bounds carry over to the solution of the relevant BSDEs. 

\begin{definition}\label{def:unifergodic}
Let $\cal M$ denote the set of probability measures on $\X$, with the topology inherited from considering them as a convex subset of $\ell_1(\X)$ (the total variation topology, with norm $\|f\|_{TV} = \sum_x |f(x)|$). We write $P_t\mu$ for the law of $X_t$ given $X_0\sim \mu$.

 We say the Markov chain $X$ is \emph{uniformly ergodic} if there exists a measure $\pi$ on $\X$, and constants $R,\rho>0$, such that
\[\sup_{\mu\in{\cal M}} \|P_t\mu-\pi\|_{TV}  \leq Re^{-\rho t} \qquad\text{for all }t.\]
In this case $\pi$ is the unique invariant measure for the chain.
\end{definition}

We now prove a few standard properties of such chains. The following lemma is simply a variant of \cite[Theorem 16.2.2(iv)]{Meyn2009}. For simplicity, we will write $E_x[\cdot]$ for $E[\cdot|X_0=x]$.

\begin{lemma}\label{ergodlem1}
 Let $X$ be a uniformly ergodic Markov chain, and let $x_C$ be an arbitrary state. Let $\tau_{C}$ be the first hitting time of $x_C$. Then for some $\beta>0$ (and hence for all $\beta$ sufficiently small),
\[G_{x_C}(\beta):=\sup_{x\in\X}E_x[e^{\beta \tau_C}]<\infty.\]
\end{lemma}

The following general continuity lemma allows us to take a bound, such as that established in Lemma \ref{ergodlem1}, and force it to converge for small $\beta$.
\begin{lemma}\label{lem:ergodboundconverge}
 Let $T$ be a random variable, and consider $G(\beta) = \sup_\nu E_\nu[e^{\beta T}]$, where $\nu$ is a parameterization of measures under which the expectation is taken (in our case, the family of initial states of the Markov chain). If there exists $\beta^*>0$ such that $G(\beta^*)\leq K$, for some $K<\infty$, then for any $\epsilon>0$,
\[G(\beta)\leq 1+\epsilon\quad \text{for all} \quad \beta\in\Big[0, \frac{\beta^*}{2}\Big(\frac{\epsilon}{K}\wedge 1\Big)\Big].\]
Note in particular that the bound does not depend on the underlying  measures, given $K$ and $\beta^*$.
\end{lemma}
\begin{proof}
Consider $G^\nu(\beta) = E_\nu[e^{\beta T}]$. For $\delta>0$ we have
\[E_\nu[Te^{\beta T}] \leq E_\nu[(\delta^{-1} e^{\delta T}) e^{\beta T}] = \frac{1}{\delta}E_\nu[e^{(\delta+\beta) T}] = \frac{1}{\delta} G^\nu(\delta+\beta)<\infty,\]
so using the dominated convergence theorem to exchange the order of integration and differentiation, in the region $\delta+\beta\leq\beta^*$ we have
\[\begin{split}\frac{d}{d\beta} G^\nu(\beta) &= \frac{d}{d\beta} E_\nu[e^{\beta T}] = E_\nu[Te^{\beta T}] \leq \frac{1}{\delta} G^\nu(\delta+\beta).\end{split}\]
Hence, fixing $\delta=\beta^*/2$, we have a (uniform) bound, on the derivative $\frac{d}{d\beta} G^\nu(\beta)$ for all $\beta\leq \beta^*/2$. Therefore, as $G^\nu(0)=1$, we know
\[G^\nu(\beta) \leq 1+\beta \Big(\frac{2}{\beta^*}G(\beta^*)\Big) \leq 1+\beta \Big(\frac{2}{\beta^*}K\Big)\]
which implies
\[G^\nu(\beta)\leq 1+\epsilon \quad \text{for all}\quad \beta \leq \frac{\beta^*}{2}\Big(\frac{\epsilon}{K} \wedge 1\Big).\]
Taking a supremum, we obtain the uniform bound $G(\beta)\leq1+\epsilon$.
\end{proof}

We now return to general properties of uniformly integrable Markov chains.

\begin{lemma}\label{ergoddiagbound}
 Let $X$ be a uniformly ergodic Markov chain, and let $|a_{xx}|:=-x^*Ax$ denote the rate of transitions out of state $x$. Then $\{|a_{xx}|\}_{x\in\X}$ is uniformly bounded above $0$.
\end{lemma}
\begin{proof}
 We know that the probability of being in the same state after a fixed timestep $\delta$ is given by $\exp(-\delta |a_{xx}|)$. Taking the $\delta$ skeleton of the chain, from \cite[Theorem 16.2.2(iii)]{Meyn2009} we see that the expected time until a transition from state $x$ is uniformly bounded with respect to $x$. By a geometric distribution argument this expectation is given by $(1-\exp(-\delta |a_{xx}|))^{-1}$, hence $|a_{xx}|$ is bounded away from zero.
\end{proof}

\begin{lemma}\label{lem:jumpsampledergodic}
 Let $X$ be a uniformly ergodic Markov chain on countable states, with bounded transition rates. Then the discrete time chain obtained by observing $X$ at each of its jump times is also uniformly ergodic.
\end{lemma}
\begin{proof}
Let $x_C$ be an arbitrary state, which is petite as our chain has countably many states. Let $N$ be the number of jumps up to hitting $x_C$, and let the $i$th jump be denoted $T_i$, with $T_0=0$. The uniform ergodicity of a Markov chain is then equivalent to stating that $\sup_xE_x[T_N]<\infty$, where $E_x$ denotes the expectation conditional on $X_0=x$. If $\alpha$ is an upper bound on the transition rates, then
\[\begin{split} E_x[T_N] &= E_x\Big[\sum_{i=1}^N(T_i-T_{i-1})]\Big]=E_x\Big[\sum_{i=1}^\infty I_{\{T_{i-1}<T_N\}} E[T_i-T_{i-1}|\F_{T_{i-1}}]\Big]\\
   & \geq E_x\Big[\sum_{i=1}^\infty I_{\{T_{i-1}<T_N\}} \alpha^{-1}\Big] =\alpha^{-1}E_x[N]
  \end{split}
 \]
Hence $\sup_x E_x[N]<\infty$, so the Markov chain generated by observations at jump times is also uniformly ergodic.
\end{proof}

To prove ergodicity properties, a standard technique is based on coupling copies of the Markov chain (see Lindvall \cite{Lindvall2002} for an overview of these methods). We here shall work with the simplest form of coupling, when we take two independent copies $X$ and $Y$ of our chain, which have distinct initial states, and we study their first meeting time. We note that throughout this section, $Y$ will be used to refer to a copy of a Markov chain, rather than to the solution of a BSDE.

\begin{lemma}\label{lem:basichitting}
 Let $X$ and $Y$ be two independent copies of the uniformly ergodic chain with  rate matrix $A$. Let $T=\inf\{t:X_t = Y_t\}$. Then there exists $\beta>0$ such that
\[G^*(\beta):=\sup_{x,y}E_{xy}[e^{\beta T}]<\infty,\] where $E_{xy}$ is the expectation conditional on $X_0=x$ and $Y_0= y$.
\end{lemma}
\begin{proof}
 Let $x_C$ be an arbitrary state, with ergodic probability $\pi(x_C)$. As the chain is uniformly ergodic, we know
\[\sup_y|P(Y_t=x_C|Y_0=y)-\pi(x_C)|\to 0 \text{ as }t\to\infty,\]
so we can find $t^*<\infty$ such that $\sup_y|P(Y_{t^*}=x_C|Y_0=y)-\pi(x_C)|< \min\{\pi(x_C), 1-\pi(x_C)\}/2,$ in particular, $P(Y_{t}=x_C)> p$ for all $t\geq t^*$, for some fixed $p>0$.

Define the following sequence of stopping times,
\[\tau_0 = \inf\{t\geq t^*: X_t = x_C\}, \qquad \tau_i = \inf\{t\geq \tau_{i-1}+t^*: X_t = x_C\}.\]
We shall show that $T_C = \inf\{\tau_i:X_{\tau_i}=Y_{\tau_i}\}$ has exponential moments, the result will then follow as $T\leq T_C$. As the two paths are independent, the definition of $p$ implies that for each $i$, $P(Y_{\tau_i}=x_C|\F_{\tau_{i-1}}) \geq p$. Hence a geometric trials argument gives
\[\begin{split} 
   E_{xy}[e^{\beta T_C}] &\leq E_{xy}\Big[\sum_{i=0}^\infty e^{\beta \tau_i} (1-p)^i\Big]\\
&= \sum_{i=0}^\infty E_{xy}\Big[\prod_{j\leq i} \Big(e^{\beta t^*}  (1-p)E_{xy}[e^{\beta(\tau_j-\tau_{j-1}-t^*)}|\F_{\tau_{j-1}+t^*}]\Big)\Big]\\
&\leq \sum_{i=0}^\infty E_{xy}\Big[\prod_{j\leq i} \Big(e^{\beta t^*} (1-p)G_{x_C}(\beta) \Big)\Big]\\
&= \sum_{i=0}^\infty (e^{\beta t^*} (1-p)G_{x_C}(\beta))^i\\
&= \Big(1- \big(e^{\beta t^*} (1-p)G_{x_C}(\beta)\big)\Big)^{-1}.
  \end{split}
\]
Hence, by taking $\beta$ small enough that $e^{\beta t^*} G_{x_C}(\beta) < (1-p)^{-1/2}$, which can be done by Lemma \ref{lem:ergodboundconverge},  we have $G^*(\beta)<\infty$.
\end{proof}

\begin{corollary}\label{cor:exphittingdiscrete}
 If $X$ and $Y$ are two independent copies of a uniformly ergodic Markov chain with bounded transition rates, and if $N$ is the total number of jumps (of both copies) until they first meet, it follows that $\sup_{x,y}E_{xy}[e^{\beta N}]<\infty$ for some $\beta>0$.
\end{corollary}
\begin{proof}
 By Lemma \ref{lem:jumpsampledergodic}, the discrete time chain obtained by observing only when one of the chains jumps is also uniformly ergodic. Hence simply apply a discrete-time version of Lemma \ref{lem:basichitting} to this setting, as the argument is identical.
\end{proof}

The following lemma provides a `lower' exponential moment.
\begin{lemma}\label{ergodlem2}
Let $X$ and $Y$ be two independent copies of a uniformly ergodic chain with bounded transition rates, and let $N$ be the total number of jumps (of both copies) until they first meet. For any fixed $\gamma\in\,]0,1[$ we know,
\[q_\gamma:=\inf_{x,y}E_{xy}[\gamma^N]>0.\]
\end{lemma}
\begin{proof}
By Lemma \ref{lem:jumpsampledergodic} the Markov chain generated by observations at jump times is also uniformly ergodic. Let $q(x,y;n) = P(N=n|X_0=x, Y_0=y)$. By Corollary \ref{cor:exphittingdiscrete} we know that there exists $\beta>0$ such that $\sum_n e^{\beta n} q(x,y;n)$ is uniformly bounded for all $x,y$. Therefore, for some constant $c$,  we know $0\leq q(x,y;n)< c e^{-\beta n}$ for all $x,y$ and $n$, hence $\sum_{n=m}^\infty q(x,y;n)< \frac{c}{1-e^{\beta}} e^{-\beta m+1}$. As we also know $\sum_n q(x,y;n)=1$, this implies that, taking $m>\log(\frac{1-e^{\beta}}{c})/\beta -1$, there is a positive lower bound on $\sum_{n=0}^{m-1} q(x,y;n)$ which is uniform in $x$ and $y$, and hence $q_\gamma>0$.
\end{proof}

\subsection{Perturbed rates}

We shall see later that the solutions to an infinite horizon BSDE can be viewed as expectations under a perturbation of the rate matrix of the underlying Markov chain. Hence, we wish to show that  if the underlying chain is uniformly ergodic then this remains the case under the perturbed measure. To do so, we use a variant of the Nummelin splitting (see \cite[Chapter 5]{Meyn2009}), which is commonly used to prove ergodicity properties.

The class of perturbations under consideration are neatly expressed through the following definition.

\begin{definition}\label{control}
 Consider $A$ and $B$ (possibly infinite) rate matrices, that is, matrices with $A_{ij}\geq 0$ for $i\neq j$ and $\sum_i A_{ij} = 0$ for all $j$, and similarly for $B$. We shall write $B\succeq A$ whenever $B-A$ is also a rate matrix.

We shall say that $B$ is \emph{controlled} by $A$ whenever there exists some $\gamma>0$ such that $\gamma A \preceq B.$  We shall say that $B$ is \emph{strictly controlled} by $A$ whenever we also know that $B-\gamma A$ has diagonal entries bounded above by $-\gamma$.
\end{definition}

Note that without loss of generality $\gamma$ can be taken to be arbitrarily small, in particular, we assume it is less than or equal to a positive lower bound on $|a_{xx}|$, which exists by Lemma \ref{ergoddiagbound}. Note also that the relation $\succeq$ defined in this way is a valid partial ordering of rate matrices, which has not, to our knowledge, been previously explored.

\begin{theorem}\label{thm:strictcontrolergodicity}
 Let $A$ and $B$ be rate matrices, and suppose $B$ is strictly controlled by $A$ with constant $\gamma$. If, under the measure induced by $A$, the process $X$ is uniformly ergodic, then $X$ is also uniformly ergodic under the measure induced by $B$. Furthermore, the constant $R$ and the rate $\rho$ of convergence to the ergodic distribution, in Definition \ref{def:unifergodic}, can be taken as functions only of $A$ and $\gamma$, and the constant $R$ can be made arbitrarily close to 1 (with a corresponding decrease in the rate $\rho$).
\end{theorem}

The proof of this theorem is the main purpose of this section. We begin by defining a splitting of the Markov chain, \`a la \cite{Meyn2009}. We assume the conditions of the theorem throughout the remainder of the section, and fix a constant $\gamma$ satisfying Definition \ref{control}.

\begin{definition}\label{defn:splitting}
 We define the \emph{split space} of $X$ to be $\check\X = \X\times\{0,1\}$. For notational simplicity, we write $\X_0:= \X\times\{0\}$ and $\X_1:= \X\times\{1\}$, so that $\check\X = \X_0\cup\X_1$.

To split a measure $\nu$ on $\X$ into a measure $\check \nu$ on $\check \X$, for $\check y\in\check\X$, if $y$ is the projection of $\check y$, then
\[\check \nu(\check y) = \begin{cases}
                 (1-\gamma)\nu(y)&  \text{for } \check y\in \X_0,\\
                 \gamma \nu(y)&  \text{for } \check y\in \X_1.\\
                \end{cases}\]
If $\check \nu$ is a signed measure with a single negative component (e.g. a column of a rate matrix), then the negative component is not split, but is assigned to one of the two corresponding states in its entirety, yielding the two splittings $\check \nu^{(0)}, \check \nu^{(1)}$.

Finally we define the split rate matrix $\B$ by
\[\B\check x = \begin{cases} 
                \check {g_x}^{(0)}&  \check x=x\times\{0\}\in \X_0,\\
		\check {a_x}^{(1)}&  \check x=x\times\{1\}\in \X_1,\\
               \end{cases}
\]
where $g_x= (1-\gamma)^{-1}(B-\gamma A)x$ and $a_x = A x$.
\end{definition}

Intuitively, transitions occur from $\check x\in \X_1$ following the vector $Ax$, and from $\check x\in\X_0$ following the vector $(1-\gamma)^{-1}(B-\gamma A)$, except that the result is randomly split between $(\X_0,\X_1)$ with probabilities $(1-\gamma, \gamma)$. Note that, as in the classical Nummelin splitting, this split chain has marginal transition matrix $B$, that is, for any measure $q$ on $\X$, we know $\int \B d\check q = \int B dq.$

Our procedure is now to consider two independent copies of our Markov chain on the split space, and to show that the first meeting time of these chains admits exponential moments. This will allow us to prove the desired uniform ergodicity estimates. We prove this in the following extended lemma.


\begin{lemma}\label{lem:exponentialhitting}
 Suppose $B$ is strictly controlled by $A$, with constant $\gamma\leq \inf_x |a_{xx}|$, and let $\check X$ and $\check Y$ be copies of the chain on the split space, as described in Definition \ref{defn:splitting}.

 Let $T=\inf\{t:\check X_t=\check Y_t\}$ and
\[H^*(\beta) = \sup_{\check x, \check y} E[e^{\beta T}|\check X_0 = \check x, \check Y_0 = \check y].\]
For any $\epsilon>0$, there exists a constant $\tilde \beta_{\gamma\epsilon}$  such that $H^*(\beta)<1+\epsilon$ for all $\beta\leq \tilde \beta_{\gamma\epsilon}$. Furthermore, $\tilde \beta_{\gamma\epsilon}$ is dependent only on $\gamma$ and $A$, in particular, it does not depend on $B$ except through $\gamma$.
\end{lemma}

\begin{proof}
We will use a renewal approach, first considering the times when a jump occurs that results in both $\check X$ and $\check Y$ being in $\X_1$, when they weren't both in $\X_1$ previously. As the two chains are independent, this will occur when one chain is already in $\X_1$, and the other jumps into $\X_1$. Let $K_t$ denote the number of times $\check X$ or $\check Y$ enters $\X_1$ up to time $t$, when the other is already in $\X_1$, and $t_k$ denote the time of the $k$th such transition.

We can bound $H^*(\beta)$ by assuming that $\check X$ and $\check Y$ first meet in $\X_1$ (that is, we ignore any prior meetings in $\X_0$). If we knew $K_T=k$,  then we know that after $t_k$, neither copy leaves $\X_1$ before their final meeting. As transitions out of $\X_1$ are independent of where in $\X_1$ the chains move, their marginal transitions on this set will follow the rate matrix $\gamma A$. We write $E^{*j}$ for the expectation conditioned on not leaving $\X_j$. Therefore, conditioning on $\check X_{t_k}, \check Y_{t_k}$ and $K_T=k$ and rescaling time, we see that
\[E[e^{\beta (T-t_k)}|\check X_{t_k}, \check Y_{t_k}, K_T=k] = E^{*1}[e^{\beta (T-t_k)}|\check X_{t_k},\check Y_{t_k}] \leq G^*(\gamma^{-1} \beta),\]
where $G^*(\beta)$ be the function defined in Lemma \ref{lem:exponentialhitting}, that is, the supremum over starting states of the moment generating function of the first hitting time \emph{on the basic (unsplit) state space $\X$ with rate matrix $A$.}

Write $E_{\check x\check y}$ for the expectation conditional on $\check X_0=\check x, \check Y_0=\check y$, and similarly $P_{\check x\check y}$ for the conditional probability. We have the bound
\begin{equation}\label{eq:boundH1}
\begin{split} 
E_{\check x\check y}[e^{\beta T}]&=E_{\check x\check y}\Big[E_{\check x\check y}[e^{\beta T}|\check X_{t_k}, \check Y_{t_k}, K_T]\Big] \\
&\leq E_{\check x\check y}\Big[\sum_{k=0}^\infty E[e^{\beta T}|\check X_{t_k}, \check Y_{t_k}, K_T=k]P_{\check x\check y}(K_T=k|\check X_{t_k}, \check Y_{t_k})\Big] \\
&\leq \sum_{k=0}^\infty E_{\check x\check y}[e^{\beta t_k}E[e^{\beta (T-t_k)}|\check X_{t_k}, \check Y_{t_k}, K_T=k]P_{\check x\check y}(K_T=k|\check X_{t_k}, \check Y_{t_k})] \\
&\leq G^*(\gamma^{-1}\beta)\sum_{k=0}^\infty E_{\check x\check y}[e^{\beta t_k}P_{\check x\check y}(K_T=k|\check X_{t_k}, \check Y_{t_k})]\\
&\leq G^*(\gamma^{-1}\beta)\sum_{k=0}^\infty (E_{\check x\check y}[e^{2\beta t_k}]E_{\check x\check y}[P_{\check x\check y}(K_T=k| \check X_{t_k}, \check Y_{t_k})^2])^{1/2}\\
&\leq G^*(\gamma^{-1}\beta)\sum_{k=0}^\infty E_{\check x\check y}[e^{2\beta t_k}]^{1/2}P_{\check x\check y}(K_T=k)^{1/2} .
  \end{split}
\end{equation}
We will seek to bound the components of the above final sum.

\textbf{Probability component} First consider $P_{\check x\check y}(K_T=k)$. Exactly as in the proof of Lemma \ref{ergodlem2}, from points $x,y\in \X$, let $q(x,y;n)$ denote the probability of a pair of independent basic (unsplit) chains meeting in precisely $n$ jumps, under the rate matrix $\gamma A$, given starting values $X_0=x, Y_0=y$. Then a geometric trials argument yields, for $\check x=x\times\{1\},\check y=y\times\{1\} \in \X_1$,
\[P_{\check x\check y}(K_T=0) = \sum_n q(x,y;n)\gamma^n.\]
By another geometric trials argument, we see that
\begin{equation}\label{eq:pbound}
\begin{split}P_{\check x\check y}(K_T=k) &= E_{\check x\check y}\Big[\Big(\sum_n q(\check X_{t_k}, \check Y_{t_k};n) \gamma^n\Big) \prod_{0\leq i<k} \Big(1- \sum_n q(\check X_{t_i}, \check Y_{t_i};n) \gamma^n\Big)\Big]\\
   &\leq (1-\gamma) (1-q_\gamma)^k,
  \end{split}
\end{equation}
where $q_\gamma := \inf_{x,y} (\sum_n q(x, y;n)\gamma^n)>0$ by Lemma \ref{ergodlem2}.

\textbf{Expectation component} Now consider $E_{\check x\check y}[e^{\beta t_k}]$. To begin, consider
\[Q(\beta):=\sup_{\check x \check y}E_{\check x\check y}[e^{\beta t_1}],\]
the mgf of the first time $\check X$ and $\check Y$ will both be in $\X_1$, following at least one of them entering $\X_0$, maximised over starting conditions. We call a jump which results in one of $\check X, \check Y$ changing between $\X_0$ and $\X_1$ a `layer shift'. Note that the two chains jump at the same time with probability zero, and so the process which counts how many of them are in $\X_0$ is skip free.

Now to bound $Q(\beta)$, we need to bound two components -- the time $\sigma_1^0$ at which the first of the two chains enters $\X_0$, and the time $t_1-\sigma_1^0$ until both chains have returned to $\X_1$. To obtain such a bound, we must also consider the possibility that, before $t_1$, both chains will enter $\X_0$. Define the stopping times
\[\begin{split}\sigma_2^i &= \inf\{t>\sigma_1^{i-1}:\check X_t,\check Y_t \in \X_0\},\\
   \sigma_1^i &= \inf\{t>\sigma_2^i: \check X_t \in \X_1 \text{ or }\check Y_t \in \X_1\},
  \end{split}
\]
so that $\sigma_i^j$ refers to the first time that precisely $i$ of the chains are in $\X_0$, given that they have both previously entered $\X_0$ precisely $j$ times.

Let
\[\begin{split}
   J_0(\beta)&:=\sup_{\check x,\check y\in\X_1}E_{\check x\check y}[e^{\beta \sigma_1^0}],\\
J_1(\beta)&:=\sup_{\check x \in\X_0,\check y\in\X_1}E_{\check x\check y}[e^{\beta \sigma_2^{1}}|\sigma_2^1<t_1],\\
J_2(\beta)&:=\sup_{\check x,\check y \in\X_0}E_{\check x\check y}[e^{\beta \sigma_1^1}],\\
J_3(\beta)&:=\sup_{\check x\in\X_0,\check y\in\X_1}E_{\check x\check y}[e^{\beta t_1}|t_1<\sigma_2^1],\\
  \end{split}
    \]
and note that the choice of whether $\check x$ or $\check y$ is in $\X_0$ or $\X_1$ is arbitrary in the definition of $J_1$ and $J_3$. This gives $J_0$ as the maximal mgf of the first time one chain enters $\X_0$; $J_1$ as the maximal mgf of the first time both chains are in $\X_0$ given one of them is now, and given this occurs before they both enter $\X_1$; $J_2$ as the first time one chain enters $\X_1$ given they are both in $\X_0$; and $J_3$ as the maximal mgf of the first time both chains are in $\X_1$ given one of them is now, and given this occurs before they both enter $\X_0$.

We can now construct another geometric trials argument. If one of $\check X$, $\check Y$ is in $\X_0$, then the next layer shift will result either in them both being in $\X_0$ or in $\X_1$. The probability of both of them moving to $\X_1$ in the next layer shift depends on the relative rates of transition, however looking at the relative probabilities of transitions we find
\[\begin{split}P(\sigma_2^1>t_1) &\geq \frac{\inf\{|b_{xx}-\gamma a_{xx}|\}\gamma}{\inf\{|b_{xx}-\gamma a_{xx}|\}\gamma + \sup\{|a_{xx}|\}(1-\gamma)}\\
   &\geq \frac{\gamma^2}{\gamma^2 + \sup\{|a_{xx}|\}(1-\gamma)}:=p
  \end{split}
 \]
and $p>0$ as $A$ is bounded. Note that $p$ does not depend on $B$. Let $D$ count the number of times both $\check X$ and $\check Y$ enter $\X_0$ before both returning to $\X_1$.  We have
\begin{equation}\label{eq:Jestimate1}\begin{split}
\sup_{\check x, \check y}E_{\check x,\check y}[e^{\beta(t_1-\sigma_1^0)}]&= \sup_{\check x, \check y}E_{\check x,\check y}[e^{\beta(t_1-\sigma_1^{D})+\beta\sum_{d=0}^{D-1}(\sigma_1^{d+1}-\sigma_1^d)}]\\
&= \sup_{\check x, \check y}E_{\check x,\check y}\Big[E[e^{\beta(t_1-\sigma_1^{D})}|\F_{\sigma_1^D},D] e^{\beta \sum_{d=0}^{D-1}(\sigma_1^{d+1}-\sigma_1^d)}\Big]\\
&\leq J_3(\beta)\sup_{\check x, \check y}E_{\check x,\check y}\Big[ e^{\beta\sum_{d=0}^{D-1}(\sigma_1^{d+1}-\sigma_1^d)}\Big].
\end{split}
\end{equation}

As $\sigma_i^D$ is not a stopping time (due to the presence of $D$), we now use a renewal argument.
\[\begin{split}
F&:= \sup_{\check x, \check y}E_{\check x,\check y}\Big[ e^{\beta\sum_{d=0}^{D-1}(\sigma_1^{d+1}-\sigma_1^d)}\Big]\\
&=\sup_{\check x, \check y}E_{\check x,\check y}\Big[ E[e^{\beta\sum_{d=1}^{D-1}(\sigma_1^{d+1}-\sigma_1^d)}|\F_{\sigma_1^1}, D>0]I_{\{D>0\}} e^{\beta(\sigma_1^{1}-\sigma_1^0)}\Big]\\
&\leq\sup_{\check x, \check y}E_{\check x,\check y}\Big[ F I_{\{D>0\}} e^{\beta(\sigma_1^{1}-\sigma_1^0)}\Big]\\
&= F \sup_{\check x, \check y}E_{\check x,\check y}\Big[I_{\{D>0\}} E[e^{\beta(\sigma_1^{1}-\sigma_2^1)}|\F_{\sigma_2^1}] e^{\beta(\sigma_2^{1}-\sigma_1^0)}\Big]\\
&\leq F J_2(\beta)\sup_{\check x, \check y}E_{\check x,\check y}\Big[I_{\{D>0\}} E_{\check x,\check y}[e^{\beta(\sigma_2^{1}-\sigma_1^0)}|D>0]\Big]\\
&\leq F J_2(\beta) J_1(\beta) \sup_{\check x, \check y}E_{\check x,\check y}[I_{\{D>0\}}]\\
&\leq F J_2(\beta) J_1(\beta) (1-p).\\
\end{split}\]
Hence provided $J_2(\beta) J_1(\beta) (1-p)<1$, we have
\begin{equation}\label{eq:Jestimate2}\sup_{\check x, \check y}E_{\check x,\check y}[e^{\beta(t_1-\sigma_1^0)}] \leq J_3(\beta)F \leq \frac{J_3(\beta)}{1- J_1(\beta) J_2(\beta) (1-p)}.\end{equation}


By Lemma \ref{lem:ergodboundconverge}, if we can uniformly bound $J_1$ and $J_2$, then by reducing $\beta$ we have a uniform bound arbitrarily close to $1$. Hence for $\beta$ sufficiently small this estimate will hold, and we can combine (\ref{eq:Jestimate1}) and (\ref{eq:Jestimate2}) to give
\begin{equation}\label{eq:ebound1}
 \sup_{\check x, \check y}E_{\check x\check y}[e^{\beta t_1}] \leq  \frac{J_0(\beta)J_3(\beta)}{1- J_1(\beta) J_2(\beta) (1-p)}.
\end{equation}
Therefore, to uniformly bound the mgf of the first return time of both chains to $\X_1$, it is enough to find uniform bounds on $J_i$ for each $i$.

\textbf{Bounding $J_i$.}
Let $\tau_i$ denote the time of the $i$th jump of the chain $\check X$, and for $j\in\{0,1\}$ let
\[R_j(i, \beta) = E\big[e^{\beta \tau_i}|\check X_t\in \X_j \text{ for all } t\in[0,\tau_i[\,\big]=: E^{*j}[e^{\beta \tau_i}],\]
so $R_j(i,\cdot)$ is the mgf of the $i$th transition of the chain, given it has not left $\X_j$. Then for $i\geq 2$,
\[\begin{split}R_j(i,\beta) &=E^{*j}[e^{\beta \tau_i}] = E^{*j}[e^{\tau_{i-1}}E^{*j}[e^{\beta \tau_i-\tau_{i-1}}|X_{\tau_{i-1}}]]\\
&\leq R_j(i-1,\beta) V_j(\beta)
  \end{split}
\]
where $V_j(\beta) = \sup_{\check x\in\X_j} E^{*j}[e^{\beta \tau_1}|\check X_0=\check x]$ is the maximum of the mgfs of the next transition. By definition, looking at the first jump we have $R_j(1,\beta)\leq V_j(\beta)$, and so by recursion, $R_j(i,\beta) \leq (V_j(\beta))^i$.

To bound $J_0$, we can instead bound the first time that $\check X$ jumps to $\X_0$, ignoring any previous jumps of $\check Y$. By a geometric trials argument, we see that
\begin{equation}\label{eq:J0}
 J_0(\beta) = \sum_{i=1}^\infty \gamma(1-\gamma)^{i-1} R_0(i,\beta) \leq \frac{\gamma V_0(\beta)}{1-(1-\gamma)V_0(\beta)}.
\end{equation}

To bound $J_1$, we need only bound the time until $\check Y$ enters $\X_0$, as the condition $\sigma_2^1<t_1$ provides an upper bound on $\sigma_2^1$, so the conditional expectation is less than the unconditional expectation. Hence, by the same geometric trials argument,
\begin{equation}\label{eq:J1}
 J_1(\beta) \leq \frac{\gamma V_0(\beta)}{1-(1-\gamma)V_0(\beta)}.
\end{equation}

To bound $J_2$, as with $J_0$, we can instead bound the first time that $\check X$ jumps to $\X_1$, ignoring any previous jumps of $\check Y$. Hence
\[J_2(\beta) = \sum_{i=1}^\infty (1-\gamma)\gamma^{i-1} R_0(i,\beta) \leq \frac{(1-\gamma) V_1(\beta)}{1-\gamma V_1(\beta)}.
\]

To bound $J_3$, as with $J_1$, we see that the condition $t_1<\sigma_2^1$ provides an upper bound on $t_1$, and so it is enough to bound the time until $\check X$ enters $\X_1$. Hence
\[J_3(\beta)\leq \frac{(1-\gamma) V_1(\beta)}{1-\gamma V_1(\beta)}.\]
Therefore to find bounds on $J_i$ for each $i$, we need only bound the right hand side of (\ref{eq:J0}) and (\ref{eq:J1}).

As jumps of Markov chains are exponentially distributed, and conditional on not leaving $\X_1$ the chains move following rate matrix $\gamma A$, we can easily see that, as $\gamma\leq \inf_x |a_{xx}|$,
\[V_1(\beta) = \sup_x \Big(1-\frac{\beta}{\gamma|a_{xx}|}\Big)^{-1} \leq\Big(1-\frac{\beta}{\gamma^2}\Big)^{-1}.\]
 As $B$ is strictly controlled by $A$, and conditionally on not leaving $\X_0$ the chains move following rate matrix $B-\gamma A$, we also have
\[V_0(\beta) = \sup_x\Big(1-\frac{\beta}{|b_{xx}-\gamma a_{xx}|}\Big)^{-1} \leq \Big(1-\frac{\beta}{\gamma}\Big)^{-1}.\]
So the right hand side of (\ref{eq:J0}) and (\ref{eq:J1}) respectively can be bounded by
\[\frac{\gamma V_0(\beta)}{1-(1-\gamma)V_0(\beta)} \leq 1+\frac{\beta}{\gamma^3-\beta}, \qquad \frac{\gamma V_0(\beta)}{1-(1-\gamma)V_0(\beta)} \leq 1+ \frac{\beta}{\gamma(1-\gamma) -\beta},\]
which are uniformly bounded for $\beta< \frac{\gamma}{2}(\gamma^2\wedge(1-\gamma))$.

Therefore, for $\beta$ sufficiently small, there is a uniform bound on (\ref{eq:ebound1}), which depends only on $\gamma$. Therefore, by Lemma \ref{lem:ergodboundconverge}, for any $\epsilon>0$,  there exists $\beta>0$ such that
\begin{equation}\label{eq:Qbound}
 Q(\beta)=\sup_{\check x, \check y} E_{\check x \check y}[e^{\beta t_1}]<1+\epsilon,
\end{equation}
and this $\beta$ depends only on $\gamma$ and $A$.

\textbf{Putting the pieces together.}
Now it is easy to see that, with $t_0:=0$,
\[E_{\check x\check y}[e^{2\beta t_k}]^{1/2} \leq E_{\check x\check y}[E[e^{2\beta \sum_{j=1}^k (t_j-t_{j-1})}|\F_{t_{j-1}}]]^{1/2} \leq E_{\check x\check y}[e^{2\beta t_1}]^{k/2}.\]
Hence, returning to (\ref{eq:boundH1}), we have
\[\begin{split}\sup_{\check x, \check y}E_{\check x\check y}[e^{\beta T}] &\leq G^*(\gamma^{-1}\beta)\sum_{k=0}^\infty \sup_{\check x, \check y}E_{\check x\check y}[e^{2\beta t_k}]^{1/2}\sup_{\check x, \check y}P_{\check x\check y}(K_T=k)^{1/2}\\
   &\leq G^*(\gamma^{-1}\beta)\sum_{k=0}^\infty (Q(2\beta))^{k/2}(1-\gamma)^{1/2} (1-q_\gamma)^{k/2}\\
&= G^*(\gamma^{-1}\beta)\frac{(1-\gamma)^{1/2}}{1-\sqrt{(1-q_\gamma)Q(2\beta)}}.
  \end{split}
\]
This is bounded whenever $Q(2\beta)<(1-q_\gamma)^{-1/2}$ and $G^*(\gamma^{-1}\beta)$ is bounded. We have just shown in (\ref{eq:Qbound}) that this can be achieved with a uniform choice of $\beta$. Hence, there exists a uniform choice of $\beta^*>0$ such that $\sup_{\check x, \check y}E_{\check x\check y}[e^{\beta T}]<K$ for all $\beta<\beta^*$. By Lemma \ref{lem:ergodboundconverge}, for any $\epsilon>0$ we can hence find $\tilde \beta_{\gamma\epsilon}$ such that
\[\sup_{\check x, \check y}E_{\check x\check y}[e^{\beta T}]\leq 1+\epsilon \qquad \text{ for all }\beta<\tilde\beta_{\gamma\epsilon}.\]
To conclude, we note that this choice of $\tilde \beta_{\gamma\epsilon}$ depends only on $\gamma$ and $A$.
\end{proof}

\begin{corollary}\label{cor:hittingestimate}
 Let $X,Y$ be two independent copies of the Markov chain on $\X$ with rate matrix $B$. Let $T$ be the first meeting time of these chains. Then for any $\epsilon>0$, there exists $\beta>0$ such that  $\sup_{x,y} E[e^{\beta T}|X_0=x, Y_0=y]\leq 1+\epsilon$.
\end{corollary}
\begin{proof}
 This follows from Lemma \ref{lem:exponentialhitting}, as the chain $X$ is simply the marginal chain of $\check X$ ignoring the splitting, and similarly for $Y$.
\end{proof}

Given this result, we can now prove Theorem \ref{thm:strictcontrolergodicity}.

\begin{proof}[Proof of Theorem \ref{thm:strictcontrolergodicity}]
For any $\mu,\nu\in{\cal M}$, consider two copies $X$, $Y$ of the chain with rate matrix $B$, where $X_0\sim \mu$, $Y_0\sim\nu$. Recall we denote by $P_t\mu$ the law of $X_t$, and $\|f\|_{TV}=\sum_x|f(x)|$ is a norm on $\cal M$.

Consider the coupled process
\[\tilde Y_t=\begin{cases} Y_t & t<T\\ X_t & t\geq T.\end{cases}\]
It is easy to see that $\tilde Y_t$ and $Y_t$ have the same law. However, this implies,
\[\|P_t\mu-P_t\nu\|_{TV}=\sum_z |P(X_t=z) - P(\tilde Y_t=z)|\leq P(T>t).\]
By Corollary \ref{cor:hittingestimate} and Chernoff's inequality, for any $\epsilon>0$ there exists $\tilde\beta_{\gamma\epsilon}>0$ dependent only on $\gamma$ and $A$  such that
\[\begin{split} P(T>t)&\leq E[e^{\tilde\beta_{\gamma\epsilon} T}|X_0\sim \mu, Y_0\sim \nu] e^{-\tilde\beta_{\gamma\epsilon} t} \\
&\leq \sup_{x,y}E[e^{\tilde\beta_{\gamma\epsilon} T}|X_0=x, Y_0=y] e^{-\tilde\beta_{\gamma\epsilon} t}\\
  &\leq (1+\epsilon) e^{-\tilde\beta_{\gamma\epsilon} t}.
  \end{split}
\]
Therefore,
\begin{equation}\label{eq:pbound2}\|P_t\mu-P_t\nu\|_{TV}\leq (1+\epsilon) e^{-\tilde\beta_{\gamma\epsilon} t}.\end{equation}
Replacing $\nu$ with $P_{s-t}\mu$ in (\ref{eq:pbound}), we see that for any $s>t$,
\[\|P_t\mu-P_s\mu\|_{TV} \leq (1+\epsilon) e^{-\tilde\beta_{\gamma\epsilon} t}\]
so $P_t\mu$ is a Cauchy sequence in $\cal M$. Hence it has a limit, $P_t\mu\to\pi$, which is invariant under $P_t$. Taking $\nu=\pi$ in (\ref{eq:pbound2}), we then see
\[\|P_t\mu-\pi\|_{TV}\leq (1+\epsilon) e^{-\tilde\beta_{\gamma\epsilon} t},\]
and so the desired convergence can be guaranteed, is independent of $\mu$, and occurs uniformly in $B$ given $\gamma$ and $A$.
\end{proof}

\section{Ergodic BSDEs}\label{sec:EBSDEs}

As we now have estimates for the uniform ergodicity of our Markov chain, we shall use these to prove the existence of solutions to Ergodic BSDEs. The techniques used here are modifications of those in \cite{Debussche2011} and \cite{Fuhrman2009}. For ease of reference, we collect the various assumptions needed in one place.

\begin{assumption}
 Let $X$ be a uniformly ergodic, time homogeneous Markov chain with bounded transition rates. Let $f:\X\times\bR^N\to\bR$ be a function satisfying the conditions of Theorem \ref{thm:discountedBSDEexist} (uniformly Lipschitz in the second component under $\|\cdot\|_{M_t}$, balanced and $\psi(\cdot, 0)$ uniformly bounded) which is also strictly balanced.
\end{assumption}

The following lemma allows us to simply connect our analysis of uniform ergodicity to the theory of BSDEs, by means of Girsanov transformations.
\begin{lemma}\label{lem:PerturbErgodic}
 Let $Z,Z'\in\bR^N$ be any two vectors, defined up to equivalence $\sim_M$, and let $f$ be our strictly balanced function. Let $\bP^x$ denote the measure under which $X$ is a Markov chain with rate matrix $A$ and initial state $x$. For each $T>0$, define the measure with density
\[\frac{d\bQ^{x,T}}{d\bP^x} = \E\Big(\int_{]0,\cdot \wedge T]} \frac{f(X_{u-}, Z) - f(X_{u-},Z')}{\|Z-Z'\|^2_{M_u}}(Z-Z')^*dM_u\Big).
\]
Then, on the interval $[0,T]$, $X$ is a time-homogeneous Markov chain under $\bQ^{x,T}$, and the corresponding rate matrix $B$ does not depend on $T$. Furthermore, $B$ is such that the Markov chain with rate matrix $B$ is uniformly ergodic, with constants $R$ and $\rho$ which do not depend on $Z$ or $Z'$.
\end{lemma}
\begin{proof}
By Lemma \ref{lem:girsanovvalid}, $\bQ^{x,T}$ is a probability measure. To show that $X$ is a time homogeneous Markov chain under $\bQ^{x,T}$, it is enough to show that an equation of the form of (\ref{ref:DMrepMarkovChain}) must hold, that is, for $t<T$ we have a rate matrix $B$ such that
\begin{equation}\label{eq:BMreptemp}
 X_t = X_0 + \int_{]0,t]} B X_{u-} du + \bQ^{x,T}\text{-martingale}
\end{equation}

However, we know that for $t<T$,
\[\begin{split}
X_t &= X_0 + \int_{]0,t]} A X_{u-} du + M_t\\
& = X_0 + \int_{]0,t]} A X_{u-} du + \int_{]0,t\wedge T]}\bigg(\frac{f(X_{u-}, Z) - f(X_{u-},Z')}{\|Z-Z'\|^2_{M_u}}\psi_u (Z-Z')\bigg) du+ \tilde M^T_t
  \end{split}
\]
and so, by uniqueness of the Doob--Meyer decomposition, we can uniquely define a matrix $B$ which satisfies (\ref{eq:BMreptemp}), through the equation
\[B {\tilde x} = A {\tilde x} + \frac{f({\tilde x}, Z) - f({\tilde x},Z')}{\|Z-Z'\|^2_{M_u}}\psi_u^{\tilde x}(Z-Z')\qquad \text{for }{\tilde x}\in\X\]
(recalling that ${\tilde x}\in\X$ is a basis vector of $\bR^N$). We know that
\[\bone^*B {\tilde x} = \bone^* A {\tilde x} + \frac{f({\tilde x}, Z) - f({\tilde x},Z')}{\|Z-Z'\|^2_{M_u}}\bone^*\psi_u^{\tilde x}(Z-Z') = 0+0,\]
so to check that $B$ is a valid rate matrix, we need to check that the only negative element can occur on the diagonal. As $f$ is balanced, we know that for any $i$ with $e_i^*A {\tilde x}>0$,
\[\frac{f({\tilde x}, Z) - f({\tilde x},Z')}{\|Z-Z'\|^2_{M_u}} ((Z-Z')^* (e_i-{\tilde x})) >-1\]
as $Z$ and $Z'$ are only defined up to equivalence $\sim_M$, hence at most up to the addition of a constant, we can assume without loss of generality that $(Z-Z')^* e_i = 0$, and so
\[\frac{f({\tilde x}, Z) - f({\tilde x},Z')}{\|Z-Z'\|^2_{M_u}} ({\tilde x}^* (Z-Z'))<1.\]
On the other hand, for this choice of representative $Z-Z'$, we then have
\[e_i^*\psi_u^{\tilde x} (Z-Z') = - (e_i^* A {\tilde x}) ({\tilde x}^* (Z-Z'))\]
and so,
\[e_i^* B {\tilde x}  = (e_i^*A {\tilde x})\bigg(1- \frac{f({\tilde x}, Z) - f({\tilde x},Z')}{\|Z-Z'\|^2_{M_u}}({\tilde x}^* (Z-Z'))\bigg)>0.\]
The fact that $B$ and $A$ also have the same pattern of zeros can be checked in a similar way, as $f$ is balanced. Therefore, $X$ is a Markov chain under $\bQ^{x,T}$ with rate matrix $B$, for all $T$.

As we also know $f$ is strictly balanced, we have exactly
\[1- \frac{f({\tilde x}, Z) - f({\tilde x},Z')}{\|Z-Z'\|^2_{M_u}}({\tilde x}^* (Z-Z')) \geq \gamma\]
for some $\gamma>0$, and hence $e_i^*(B-\gamma A){\tilde x}\geq0$ for all $e_i\neq {\tilde x}$. Hence (as the columns all sum to zero) $B-\gamma A$ is also a rate matrix, and hence $B$ is controlled by $A$. We also have
\[{\tilde x}^* B {\tilde x} = - \sum_{i:e_i\neq {\tilde x}} e_i^* B X_{u-} \leq  - \gamma \sum_{i:e_i\neq {\tilde x}} e_i^* A {\tilde x} = \gamma {\tilde x}^* A {\tilde x},\]
so
\[({\tilde x}^* B {\tilde x}) - \frac{\gamma}{2} ({\tilde x}^* A {\tilde x}) \leq \frac{\gamma}{2} ({\tilde x}^* A {\tilde x}).\]
As $X$ is uniformly ergodic under the initial measure, ${\tilde x}^* A {\tilde x}$ is bounded away from zero, by Lemma \ref{ergoddiagbound}. Therefore, for $\gamma^* < \gamma\big(\frac{\inf_{\tilde x}|{\tilde x}^* A {\tilde x}|\wedge 1}{2}\big) $, we see that $B$ is strictly controlled by $A$ with constant $\gamma^*$, independent of $Z$ and $Z'$.

As we know that $B$ is strictly controlled by $A$ with a fixed constant $\gamma^*$, we can apply Theorem \ref{thm:strictcontrolergodicity} to show that there exist constants $R, \rho$, independent of $Z$ and $Z'$, such that $X$ is uniformly ergodic under $\bQ^{x, T}$ with constants $R$ and $\rho$.
\end{proof}

We now show that, for some sequence $\alpha_n$, the solutions to the discounted BSDEs converge in an appropriate sense.
\begin{lemma}\label{lem:Discconverg}
Let $(Y^\alpha, Z^\alpha)$ be the unique bounded adapted solution to the discounted BSDE
\[Y^\alpha_T=Y^\alpha_t - \int_{]t,T]} (-\alpha Y^\alpha_{u-} + f(X_{u-}, Z^\alpha_u)) du + \int_{]t,T]} (Z^\alpha_u)^*dM_u, \qquad t\leq T<\infty,\]
and let $x_0\in\X$ be an arbitrary state. By Corollary \ref{cor:markoviansolutions}, we know there exists a function $v^\alpha:\X\to\bR$ such that $Y^\alpha_t=v^\alpha(X_t)$. Then there exists a bound $C'<\infty$ such that
\[|v^{\alpha}(x)-v^{\alpha}(x_0)|<C', \qquad \alpha |v^\alpha(x)| < C'\]
uniformly in $x$ and $\alpha$. Hence there exists a sequence $\alpha_n \to 0$ such that
\[(v^{\alpha_n}(x)-v^{\alpha_n}(x_0)) \to v(x) \quad \text{ and } \quad \alpha_n v^{\alpha_n}(x) \to \lambda \qquad \text{ for all }x\in\X,\]
for some bounded function $v:\X\to\bR$ and some $\lambda\in\bR$.
\end{lemma}
\begin{proof}
 Let $C$ be a bound on $|f(\cdot, 0)|$. Then we see from Theorem \ref{thm:discountedBSDEexist} that $|v^\alpha(\cdot)|\leq C/\alpha$.  We first need to convert this into a bound which is uniform in $\alpha$. To do so, we will view the process $Y^\alpha$ under a perturbed measure, and use our ergodicity result to control its growth.

We first note that, by Corollary \ref{cor:markoviansolutions}, the solution process $Z^\alpha$ is constant, up to equivalence $\sim_M$.

For each $T>0$, define the measure with density
\[\frac{d\bQ^{x,\alpha, T}}{d\bP^x} = \E\Big(\int_{]0,\cdot \wedge T]} \frac{f(X_{u-}, Z^\alpha_u) - f(X_{u-},0)}{\|Z^\alpha_u\|^2_{M_u}}(Z^\alpha)^*dM_u\Big).
\]
By Lemma \ref{lem:PerturbErgodic}, under this measure $X$ is still a uniformly ergodic Markov chain, with constants $R,\rho$ which do not depend on $\alpha$ (as $\alpha$ only affects $Z^\alpha$).

Using Girsanov's theorem in the form of Lemma \ref{lem:girsanovvalid}, we can verify that for any $x, \alpha, T$,
\[v^\alpha(x) = E_{\bQ^{x,\alpha, T}}\Big[ e^{-\alpha T}v^\alpha(X_T) + \int_{]0,T]} e^{-\alpha u} f(X_{u-}, 0)du\Big].\]
As $|v(x)|\leq C/\alpha$, letting $T\to\infty$ we obtain
\[v^\alpha(x) = \lim_{T\to\infty}E_{\bQ^{x,\alpha, T}}\Big[\int_{]0,T]} e^{-\alpha u} f(X_{u-}, 0)du\Big].\]

Therefore, for any $x,x'\in\X$, we have
\[\begin{split}|v^\alpha(x) - v^\alpha(x')| &= \Big|\lim_{T\to\infty}E_{\bQ^{x,\alpha, T}}\Big[\int_{]0,T]} e^{-\alpha u} f(X_{u-}, 0)du\Big]\\&\qquad-\lim_{T\to\infty}E_{\bQ^{x',\alpha, T}}\Big[\int_{]0,T]} e^{-\alpha u} f(X_{u-}, 0)du\Big]\Big|\\
   &= \Big|\lim_{T\to\infty}\int_{]0,T]} e^{-\alpha u} \int_\X f(X_{u-}, 0) (dP_u\delta_x - dP_u\delta_{x'})du\Big|\\
&\leq C \Big|\lim_{T\to\infty}\int_{]0,T]} e^{-\alpha u} \|P_u\delta_x - P_u\delta_{x'}\|_{TV}du\Big|\\
&\leq CR \Big|\lim_{T\to\infty}\int_{]0,T]} e^{-\alpha u} e^{-\rho u} du\Big|\\
&= CR (\alpha+\rho)^{-1}
  \end{split}
\]
And so we have a uniform bound $|v^\alpha(x) - v^\alpha(x')|\leq CR \rho^{-1}$. In addition, we recall that $|\alpha v^\alpha(x)|\leq C$.

Given this, we can use a diagonal procedure to construct a sequence $\alpha_n\downarrow 0$ such that
\[\begin{split}
   \alpha_nv^{\alpha_n}(x_0) &\to \lambda\\
 v^{\alpha_n}(x) - v^{\alpha_n}(x_0) &\to v(x) \qquad \text{ for all }x\in\X
  \end{split}\]
for a function $v:\X\to\bR$ (recall that $\X$ is at most countable). Note that, for $C'=CR\rho^{-1}$,  we have the bounds $|\lambda|\leq C'$ and $|v(x)|\leq C'$.

Finally, we notice that, for an arbitrary $x\in\X$,
\[\alpha_nv^{\alpha_n}(x) = \alpha_nv^{\alpha_n}(x_0) + \alpha_n(v^{\alpha_n}(x)- v^{\alpha_n}(x_0)) \to \lambda+0\]
so the convergence of this sequence to $\lambda$ holds for all $x$.
\end{proof}

\begin{theorem}\label{thm:EBSDEexist}
Let $v$, $\lambda$ be as constructed in Lemma \ref{lem:Discconverg}. The triple $(Y, Z, \lambda)$, where
\[Y_t:= v(X_t),\quad e_i^*Z_t= v(e_i),\] is the unique bounded Markovian solution, with $v(x_0)=0$, to the Ergodic BSDE
\[Y_t = Y_T + \int_{]t,T]}  [f(X_{u-}, Z_u) - \lambda]du + \int_{]t,T]}  Z_u^*dM_u.\]
Any other bounded solution $(Y',Z', \lambda')$ satisfies $\lambda=\lambda'$, and any other bounded Markovian solution $(Y',Z', \lambda')$ satisfies $Y'_t=Y_t+c$ for some $c\in\bR$, and $Z'\sim_M Z$.
\end{theorem}

\begin{proof}
 Let $\alpha_n$ be the sequence constructed in Lemma \ref{lem:Discconverg}. As $Z^{\alpha_n}$ is only defined up to constant shifts, we know that $Y^{\alpha_n}_t:=v^{\alpha_n}(X_t)$, $e_i^*Z^{\alpha_n}_t := v^{\alpha_n}(e_i)-v^{\alpha_n}(x_0)$ solves the discounted BSDE
\[Y^{\alpha_n}_t = Y^{\alpha_n}_T + \int_{]t,T]} \big(f(X_{u-}, Z^{\alpha_n}_u) -{\alpha_n} Y^{\alpha_n}_{u-}\big) du + \int_{]t,T]} (Z^{\alpha_n}_u)^*dM_u.\]
Note that $Z^{\alpha_n}_t$ is constant in $t$, and by the bound established in Lemma \ref{lem:Discconverg}, $|e_i^*Z^{\alpha_n}_t|$ is uniformly bounded. Hence, as the transition rates in $A$ are bounded, we also know $\|Z^{\alpha_n}\|^2_{M_t}$ is uniformly bounded, and dominated convergence yields $\lim (\int Z^{\alpha_n} dM )= \int (\lim Z^{\alpha_n}) dM$. Therefore
\[\begin{split} v(X_t) &= \lim_n (v^{\alpha_n}(X_t) - v^{\alpha_n}(x_0))\\
   &= \lim_n (v^{\alpha_n}(X_T)- v^{\alpha_n}(x_0)) + \lim_n \int_{]t,T]} \big(f(X_{u-}, Z^{\alpha_n}) -{\alpha_n} v^{\alpha_n}(X_{u-})\big) du\\
&\qquad + \lim_n \int_{]t,T]} (Z^{\alpha_n})^*dM_u\\
   &= v(X_T) + \int_{]t,T]} \lim_n \big(f(X_{u-}, Z^{\alpha_n}-v^{\alpha_n}(x_0)) -\lim_n {\alpha_n} v^{\alpha_n}(X_{u-})\big) du\\
&\qquad + \int_{]t,T]} \lim_n (Z^{\alpha_n}-v^{\alpha_n}(x_0))^*dM_u\\
   &= v(X_T) + \int_{]t,T]} \big(f(X_{u-}, Z) -\lambda\big) du+ \int_{]t,T]} Z^*dM_u,\\
  \end{split}
\]
that is, $(v(X_t), v(\cdot), \lambda)$ is a solution to the EBSDE.

To show the solution is unique, suppose $(Y', Z', \lambda')$ is another bounded solution. Let $\tilde Y = Y-Y'$, $\tilde Z=Z-Z'$ and $\tilde \lambda=\lambda-\lambda'$. Then
\[\begin{split}
   \tilde Y_t = \tilde Y_T + \int_{]t,T]}  [f(X_{u-}, Z) - f(X_{u-}, Z'_u)- \tilde \lambda]du + \int_{]t,T]}  \tilde Z_u^*dM_u
  \end{split}
\]
and defining the measure $\bQ_1^T$ as in (\ref{eq:Q1})
\[\frac{d\bQ^T_1}{d\bP^x} = \E\Big(\int_{]0,\cdot \wedge T]} \frac{f(X_{u-}, Z) - f(X_{u-},Z'_u)}{\|\tilde Z\|^2_{M_u}}(\tilde Z_u)^*dM_u\Big),
\]
 we see that
\[\tilde \lambda = T^{-1} E_{\bQ^T_1}\big[\tilde Y_T - \tilde Y_0].\]
As $\tilde Y$ is uniformly bounded, taking $T\to\infty$ we see $\tilde \lambda =0$, that is, $\lambda=\lambda'$. Substituting back, we see that $\tilde Y_0 = E_{\bQ^T_1}\big[\tilde Y_T]$ for all $T$.

From Lemma \ref{lem:PerturbErgodic}, if $Y'$ is a bounded Markovian solution with $Y'_t=v'(X_t)$, we see that $X$ is still a time homogenous uniformly ergodic Markov chain under $\bQ^T_1$, on the interval $[0,T]$. The family of measures $\bQ^T_1$ is then consistent, in that $\bQ^T|_{\F_T}= \bQ^{T'}|_{\F_T}$ for all $T\leq T'$. Therefore, conditioning on $X_0=x$, we can extend them to a single measure $\bQ^x$ under which $X$ is a time homogenous uniformly ergodic Markov chain for all times, and $\bQ^x|_{\F_T}=\bQ^T_1|_{\F_T, X_0=x}$ for all $T$. If $\tilde\pi$ is the ergodic measure of $X$ under $\bQ^x$, we then have
\[v(x)-v'(x) = \tilde Y_0= \lim_T E_{\bQ^x}[\tilde Y_T] = \int_\X \big(v(y)-v'(y)\big) d\tilde\pi(y),\]
and we see that the right hand side is independent of $x$, that is, $v'(x)=v(x)+c$ for some constant $c\in\bR$, for all $x\in\X$. The equivalence $\|Z-Z'\|_{M_t}^2=0$ then follows from Corollary \ref{cor:markoviansolutions}.
\end{proof}
\begin{corollary}
 The sequences $\alpha_nv^{\alpha_n}$ and $v^{\alpha_n}(x) - v^{\alpha_n}(x_0)$ constructed in Lemma \ref{lem:Discconverg} converge for any choice of sequence $\{\alpha_n\downarrow 0\}$.
\end{corollary}
\begin{proof}
 Suppose this were not the case. Then we could construct two sequences with distinct limits, both of which would yield bounded Markovian solutions to the EBSDE. However we have shown that there is only one such solution, which is a contradiction.
\end{proof}

\begin{corollary}\label{cor:simplify}
The value $\lambda$ in the EBSDE solution $(Y, Z, \lambda)=(v(X_t), v, \lambda)$ satisfies
 \[\lambda = \int_\X f(y, v) d\pi(y),\]
and hence, if we define a signed measure $\mu^x$ on $\X$ by
$\mu^x(y) = \int_{]0,\infty[} (\bP^x_t(y) - \pi(y)) dt,$
(which is well defined due to the uniform ergodicity of $X$ under $\bP$) we have
\[v(x) = c+ \int_\X f(y,v) d\mu^x(y)\]
for some $c=\int_{\X} v(x)d\pi\in\bR$.
\end{corollary}
\begin{proof}
 The invariance of the ergodic distribution $\pi$ implies that, for any fixed time $T$, any $g:\X\to\bR$,
\[\int_\X E_{\bP^x}[g(X_T)] d\pi(x)=\int_\X E_{\bP^x}[g(X_{T-})] d\pi(x) = \int_\X g(x) d\pi(x).\]
We write
\[v(x) = E_{\bP^x}\Big[v(X_T) + \int_{]0,T]} \big(f(X_{u-}, v) -\lambda\big) du\Big].\]
so by the stated invariance property,
\[\begin{split}
   \int_\X v(x)d\pi(x) &= \int_\X E_{\bP^x}[v(X_T)] d\pi(x) + \int_\X E_{\bP^x}\Big[\int_{]t,T]} \big(f(X_{u-}, v) -\lambda\big) du\Big]d\pi(x)\\
 &= \int_\X v(x) d\pi(x) + \int_{]t,T]} \Big[\int_\X E_{\bP^x}[f(X_{u-}, v)]d\pi(x)\Big] du -T\lambda \\
 &= \int_\X v(x) d\pi(x) + \int_{]t,T]} \Big[\int_\X f(x, v)d\pi(x)\Big] du -T\lambda \\
 &= \int_\X v(x) d\pi(x) + T\int_\X f(x, v)d\pi(x)  -T\lambda\\
  \end{split}
\]
and rearrangement yields the first result.

Consequently, we also have
\[\begin{split}
  v(x) &= \lim_{T\to\infty}E_{\bP^x}\Big[v(X_T) + \int_{]0,T]} \big(f(X_{u-}, v) -\lambda\big) du\Big]\\
&= \int_\X v(y) d\pi(y) + \lim_{T\to\infty}\int_{]0,T]} \Big(\int_\X f(y, v)d\bP^x_u(y)  -\int_\X f(y, v)d\pi(y)\Big) du\\
&= \int_\X v(y) d\pi(y) + \int_{]0,\infty[}\Big( \int_\X f(y, v)d(\bP^x_u(y)-\pi(y))\Big) du\\
&= \int_\X v(y) d\pi(y) + \int_\X f(y, v)d\mu^x(y).\\
  \end{split}
\]

\end{proof}

\begin{remark}
 If the Lipschitz constant of $f$ were sufficiently small, when compared with the coefficients in the ergodicity of $X$, the above corollary would provide a simple direct proof of existence for EBSDEs. Furthermore, while we do not have a comparison theorem for the $Y$ component of BSDEs, from this result, if $c$ is fixed, we can differentiate $v(x)$ in terms of $f(y, v(\cdot))$, and (again given a small enough Lipschitz constant) solve the resultant implicit equation. From this analysis, we can see that, if $f$ is sufficiently small, then $v(x)$ is monotone increasing in $f(x,\cdot)$.
\end{remark}

\begin{remark}
 This representation of $v$ in fact provides an intuitive meaning for $v$. We see here that, up to the addition of a constant, $v(x)$ is given by the cost $f$, integrated through time with respect to the `deviation from ergodicity' measure $\mu^x$. Therefore, it is natural to think of $v$ as giving a `short-run additional expected cost' term, while $\lambda$ gives the long-run ergodic cost.
\end{remark}

We now give a variant of the comparison theorem for EBSDEs, considering the $\lambda$ component of the solution.
\begin{theorem}\label{thm:compthmlambda}
 Let $f$ and $f'$ be two strictly balanced drivers, and $(Y, Z, \lambda)$ and $(Y', Z', \lambda')$ the corresponding EBSDE solutions. Then if $f(x,z)\geq f'(x,z)$ for all $x\in\X$, $Z\in\bR^N$, then $\lambda\geq\lambda'$.
\end{theorem}
\begin{proof}
 Take the measure with density
\[\frac{d\bQ^{x,T}}{d\bP^x} = \E\Big(\int_{]0,\cdot \wedge T]} \frac{f'(X_{u-}, Z) - f'(X_{u-},Z')}{\|Z-Z'\|^2_{M_u}}(Z-Z')^*dM_u\Big),\]
and the corresponding rate matrix $B$ and resultant ergodic measure $\pi_B$. Then, if $Y_t=v(X_t)$ and $Y'_t=v'(X_t)$, integrating the difference of the EBSDEs yields
\[\begin{split}
   &\int v(x)-v'(x) d\pi_B \\
&= \int_\X E_{\bQ^{x,T}}\Big[ Y_T - Y'_T+ \int_{]0,T]} \big( f(X_{u-}, Z) - f'(X_{u-},Z)\big)dt - T(\lambda-\lambda')\Big]d\pi_B\\
   &= \int_\X \big(v(x)-v'(x)\big) d\pi_B + \int_{]0,T]} \Big(\int_\X f(x, Z) - f'(x,Z)d\pi_B \Big)dt - T(\lambda-\lambda')\\
  \end{split}
\]
and so
\[(\lambda-\lambda')= \int_\X \big(f(x, Z) - f'(x,Z)\big)d\pi_B\geq 0.\]
\end{proof}

\begin{remark}
 We note that we have also shown that, given that $f$ is strictly balanced and $A$ is the generator of a uniformly ergodic chain, the vector equation
\[0 =  f (\cdot, v) +A^*v-\lambda\bone\]
admits a unique solution $(v(\cdot), \lambda)\in\bR^N\times\bR$, and that this solution is approximated by the equations considered in Corollary \ref{cor:markoviansolutions}. This is without any further reference to the monotonicity of $f$. This can be compared with the classical equation for ergodic cost, which is simply
\[0= f(\cdot) + A^*v-\lambda\bone.\]
\end{remark}

\begin{remark}\label{rem:numerics}
 From a numerical perspective, as we expect that $f(\cdot,v) + A^*v$ will typically be `approximately monotone' in $v$, solving this equation iteratively, for example through the equations
\[\begin{split}
\lambda_{n}& = N^{-1} \bone^*\Big(f(\cdot, v_n)+A^*v_n\Big)\\
   v_{n+1}&= (A^*)^+\Big(\lambda_n\bone- f(\cdot, v_n)\Big)
  \end{split}\]
where $(A^*)^+$ denotes the Moore-Penrose pseudoinverse of $A^*$, will often provide a simple numerical scheme in the finite state case. 

In the infinite state case, a Monte-Carlo approach can be used to approximate the equations given by Corollary \ref{cor:simplify}, however the convergence of such a scheme lies beyond the scope of this paper.
\end{remark}

\section{Applications and Examples}\label{sec:example}
\subsection{Rate uncertainty}
As a first example, we consider the rate matrix $A$ constructed (\ref{eq:pathA}), and the driver
\[f(x, v) = I_{x\in\zeta}+ \min_{r\in[1/\beta, \beta]}\{(r-1)v^*Ax\},\]
where $\zeta\subseteq \X$ and $\beta>1$. Essentially, this driver attempts to determine the ergodic probability of being in the set $\zeta$, however it introduces an uncertainty about the overall transition rate of the chain, by scaling the rates up or down with the parameter $\beta$, so as to minimise the probability of being in $\zeta$. A related BSDE over finite horizon, for a different choice of $A$, was considered in \cite{Cohen2011b}. 
 
For $\beta=2$, we solve the EBSDE numerically using the simple algorithm suggested in Remark \ref{rem:numerics}. We list all possibilities for $\zeta$ up to symmetries, and choose $v$ such that its values in each state sum to zero. For comparison, we also list the ergodic probability associated with the set $\zeta$.
\begin{center}
\begin{tabular}{c|rrrr|cc}
 $\zeta$ & $v(e_1)$ & $v(e_2)$ & $v(e_3)$ & $v(e_4)$ & $\lambda$& $\pi(\zeta)$\\
\hline
$\emptyset$ &$ 0$&$0$&$0$&$0$&$0$& $0$\\
$\{e_1\}$ &$    0.1207 $ &$  -0.0172 $ &$  -0.0517  $ &$ -0.0517$ &$ 0.0345$ & $0.125$\\
$\{e_2\}$ & $-0.0652  $ & $  0.1087  $ & $ -0.0217  $ & $ -0.0217$ & $0.1304$& $0.375$\\
$\{e_1, e_2\}$&$ 0.1000   $&$ 0.1000 $&$  -0.1000 $&$  -0.1000$&$0.2000$& $0.500$\\
$\{e_1, e_3\}$&$0.1500  $&$ -0.0500  $&$  0.0500  $&$ -0.1500$&$ 0.2000$& $0.500$\\
$\{e_1, e_4\}$ &$     0.0769  $&$ -0.0769  $&$ -0.0769  $&$  0.0769$&$ 0.0769$& $0.250$\\
$\{e_2, e_3\}$ &$   -0.1429  $&$  0.1429  $&$  0.1429 $&$  -0.1429 $&$ 0.4286$& $0.750$\\
$\{e_1, e_2,e_3\}$&$     0.1364  $&$  0.1364 $&$   0.0455  $&$ -0.3182 $&$ 0.6364$& $0.875$\\
$\{e_1, e_2, e_4\}$ &$    0.0294 $&$   0.0294  $&$ -0.1471 $&$   0.0882 $&$ 0.2941$& $0.625$\\
$\{e_1, e_2, e_3, e_4\}$&$     0  $&$   0  $&$   0  $&$   0$&$1$& $1$\\
\end{tabular}
\end{center}

Note that, in the classical case without the rate uncertainty term, we would have $\lambda_\zeta = \pi(\zeta)$. This is clearly not the case here, due to the nonlinearity introduced by the rate uncertainty. In fact, we can see here that $\lambda$ defines, in some sense, an `ergodic capacity' for the chain (however one can verify that the EBSDE solution is generally not given by the Choquet integral with respect to this capacity). Note further that while the ergodic probability does not reveal the asymmetry between the states of the chain, the EBSDE is affected by this in a nontrivial way.

\subsection{Classical Optimal Ergodic Control}\label{sec:control}
We now give a more abstract example, indicating how an ergodic control problem can be seen in this framework. Problems of this sort have been extensively considered, see for example the classical paper of Kakumanu \cite{Kakumanu1972}, or more recent work by Guo and Hern\'andez-Lerma \cite{Guo2001}.

Consider the problem of minimizing
\[J(x, U) = {\lim\sup}_{T\to\infty} T^{-1} E^{U}_x\Big[\int_{]0,T]} L(X_{u-}, U_t) dt\Big]\]
where
\begin{itemize}
 \item $\U$ is a space of controls, which is a separable metric space,
 \item $U$ is a $\U$ valued predictable process,
 \item $L:\X\times\U\to\bR$ is a bounded measurable cost function,
 \item $E^{U}_x$ is the expectation under which at time $t$, for the path $\omega$, $X$ jumps from state $e_i$ to state $e_j$ at a rate $e_j^* A^{U_t(\omega)}e_i$, for some measurable matrix valued function $A^{(\cdot)}:\U\to\text{rate matrices}$, and $X_0=x$,
 \item for some $\gamma>0$, for all $u \in\U$, the matrices $A^{u}$ are uniformly bounded and strictly dominated by $A$ with constant $\gamma$ (in the sense of Definition \ref{control}), for some reference rate matrix $A$ under which $X$ is a uniformly ergodic Markov chain.
\end{itemize}
We shall write $E$ for the expectation under which $X$ is a uniformly ergodic Markov chain with rate matrix $A$.

We define the Hamiltonian
\begin{equation}\label{eq:hamiltonian}
 f(x,z) = \inf_{u\in\U}\{L(x,u) + z^* (A^u-A)x\}.
\end{equation}
We notice that $f(\cdot, 0)$ is bounded, and by the assumption that the $A^u$ are strictly controlled by $A$, it is easy to show that $f$ is a Lipschitz function in $z$ under the $\|\cdot\|_{M_t}$ norm, and that $f$ is strictly balanced (by testing with $z$ equal to each basis vector). Therefore, the EBSDE with driver $f(x,z)$ admits a unique bounded Markovian solution $(\bar Y_t,\bar Z_t,\bar\lambda) = (v(X_t), v, \bar\lambda)$, where $v$ is the vector with components $v(e_i)$.

If this infimum is attained, then there exists (assuming the continuum hypothesis, by McShane and Warfield \cite{McShane1967}) a measurable function $\kappa:\X\times\bR^N\to\U$ such that
\[f(x,z) = L(x,\kappa(x,z)) + z^*(A^{\kappa(x,z)}-A)x.\]
 We then have the following theorem.

\begin{theorem}\label{thm:optimalcontrol}
 In the setting described above, let $(Y, Z, \lambda)$ be any (possibly non-Markovian) bounded solution to the EBSDE (\ref{eq:EBSDE}) with driver $f$. Then the following hold.
\begin{enumerate}[(i)]
 \item For an arbitrary control $U$ we have $J(x,U)\geq \lambda=\bar\lambda$, and equality holds if and only if
\[L(X_{t-}, U_t) + Z_t^* (A^{U_t}-A)X_{t-}= f(X_{t-}, Z_t)\qquad d\bP\times dt-a.e.\]
\item If the infimum is attained in (\ref{eq:hamiltonian}), then the control $\bar U_t = \kappa(X_t, Z_t)$ verifies $J(x, \bar U)= \bar \lambda$. 
\end{enumerate}
In particular, for the bounded Markovian solution to the EBSDE, in terms of the vector $v$ the following hold.
\begin{enumerate}[(i)]
\setcounter{enumi}{2}
 \item For arbitrary controls $U$ we have $J(x, U) = \bar \lambda$ if and only if
\[L(X_{t-}, U_t) + v^*(A^{U_t}-A)X_{t-}= f(X_{t-}, v)\qquad d\bP\times dt-a.e.\]
\item If the infimum is attained in (\ref{eq:hamiltonian}), then the control $\bar U_t = \kappa(X_{t-}, v)$ verifies $J(x, \bar U)= \bar \lambda$, that is, we have an optimal feedback control.
\end{enumerate}
\end{theorem}
\begin{proof}
  That $\lambda=\bar\lambda$ is a consequence of Theorem \ref{thm:EBSDEexist}, and points (iii) and (iv) follow directly from (i) and (ii).

To show (i), for $T>0$ define the measures,
\[\frac{d\bQ^{x,T}}{d\bP^x} = \E\Big(\int_{]0,\cdot \wedge T]} \frac{Z_s^* (A^{U_s}-A)X_{s-} }{\|Z_s\|^2_{M_s}}Z_s^*dM_s\Big).\]
As $(Y,Z,\bar\lambda)$ is a solution to the EBSDE, we have
\[\begin{split}
   \bar\lambda&=\frac{1}{T} E_{\bQ^{x,T}}[Y_T-Y_0] + \frac{1}{T} E_{\bQ^{x,T}}\Big[\int_{]0,T]} L(X_{s-}, U_s) ds\Big]\\
&\qquad +\frac{1}{T} E_{\bQ^{x,T}}\Big[\int_{]0,T]} f(X_{s-}, Z_s) - Z_s^*(A^{U_s}-A)X_{s-} - L(X_{s-}, U_s) ds\Big],
  \end{split}
\]
and hence, as $f$ is an infimum over the controls,
\[\bar\lambda + \frac{1}{T} E_{\bQ^{x,T}}[Y_0-Y_T] \leq \frac{1}{T} E_{\bQ^{x,T}}\Big[\int_{]0,T]} L(X_{s-}, U_s) ds\Big].\]
As $Y_0-Y_T$ is uniformly bounded, taking a limit we see that
\[\bar\lambda \leq {\lim\sup}_{T\to\infty} \frac{1}{T} E_{\bQ^{x,T}}\Big[\int_{]0,T]} L(X_{s-}, U_s) ds\Big] = J(x,U).\]
Conversely, if the infimum is attained, we have $f(X_{s-}, Z_s)= L(X_{s-}, \bar U_s)+ Z_s(A^{\bar U_s}-A)X_{s-}$ for some $\bar U_s = \kappa(X_{s-},Z_s)$. Then equality holds throughout and we see that
\[\bar\lambda = {\lim\sup}_{T\to\infty} \frac{1}{T} E_{\bQ^{x,T}}\Big[\int_{]0,T]} L(X_{s-}, \bar U_s) ds\Big] = J(x,\bar U).\]
\end{proof}

\begin{remark}
From this analysis, we also see that one could equally define the cost functional $J(x,U)$ as the $\lim\inf$, rather than the $\lim\sup$, and the same conclusions would hold, with the same optimal cost $\bar\lambda$. This is simply because the function under consideration converges for an optimal policy, as given by the EBSDE solution.
\end{remark}

\subsection{Risk averse control}
The strength of the EBSDE approach is, however, not made manifest by the classical linear setting. Through the use of EBSDEs, one can happily consider nonlinear examples, in particular, when the expectation $E_x^U$ is replaced by a dynamically consistent nonlinear expectation, in the sense of Peng \cite{Peng1997}.

For example, one could consider a form of risk averse ergodic control under uncertainty. In this case, we are in a similar setting to above, but the control $U_s$ does not yield a unique rate matrix $A^{U_s}$, rather a family of matrices $\{A^{U_s, w}\}_{w\in W_{U_s}}$, where $W_{U_s}$ is an index set, and the $A^{U_s,w}$ are still all uniformly strictly controlled by a single rate matrix $A$, as in the previous section.

To ensure dynamic consistency, we assume that $W_u$ is such that if $w$ and $w'$ are both in $W_{U_s}$, then there is an element $w''\in W_{U_s}$ such that the matrix obtained by interchanging arbitrary columns of $A^{u,w}$ and $A^{u, w'}$ is equal to $A^{u, w''}$.

In this case, we could attempt to minimise the maximum over the relevant measures (or some other nonlinear functional) corresponding to the value functional
\[J(x,u) = {\lim\sup}_{T\to\infty} T^{-1} \sup_{w\in W_{U_s}}E^{U,w}_x\Big[\int_{]0,T]} L(X_t, U_t) dt\Big],\]
(or equivalently with a $\lim\inf$). This can be directly treated by our setting, simply by taking the non-concave Hamiltonian
\[f(u;x,z) = \inf_{u\in\U}\big\{L(x, u)  + \sup_{w\in W_u}\{z^*(A^{u,w}-A)x\}\big\}\]
and solving the corresponding nonlinear EBSDE. As above, one can verify that $J(x,u)$ is then the $\lambda$ component of the solution to the EBSDE with driver $f(u,\cdot,\cdot)$.  Extending Theorem \ref{thm:optimalcontrol} to this setting is then a straightforward task.

\begin{remark}
 In \cite{Cohen2011a}, conditions are given under which a dynamically consistent nonlinear expectation in a general setting can be represented by means of a BSDE with a balanced driver. Suppose each choice of policy yielded a nonlinear expectation. For each policy which makes $X$ a homogeneous Markov process (in the sense that the conditional nonlinear expectation of $\phi(X_t)$ given $\F_s$ is a function of $X_s$ and $t-s$, for any measurable $\phi$ and any $s\leq t$), the drivers of these BSDEs can be written in the form $g_u:\X\times\bR^N\to\bR$, and assume that these drivers are strictly balanced. Given this representation, the Hamiltonian for our cost minimization problem has the generic form
\[f(u; x, z) = \inf_{u\in\U} \{L(x, u) + g_u(x,z)\}.\]
\end{remark}

\section{Conclusions}\label{sec:conclusion}

We have seen that, for uniformly ergodic Markov chains, EBSDEs with strictly balanced drivers admit unique bounded Markovian solutions. The methods used to determine this result, while based on previous work on EBSDEs, require a different approach to the ergodicity of the underlying process, due to the presence of jumps.

In deriving this, we have constructed a partial ordering of the rate matrices, and shown that any chain bounded below by a uniformly ergodic chain must also be uniformly ergodic, and that the rate coefficients can be uniformly bounded. In some sense, this result is similar to recent work by Galtchouk and Pergamenshchikov \cite{Galtchouk2012}, who study geometric ergodicity properties of general Markov processes, under an assumption of a uniform Lyapunov function. On the other hand, our method is better suited to the study of EBSDEs, where the perturbation of the rate matrix arises directly from the driver of the BSDE. Future work may allow a weakening of our assumption of uniform ergodicity to a form of geometric ergodicity, however we expect that this will require some restriction of the class of Markov chains (for example, to stochastically monotone chains). Such techniques have been used for ergodic costs from a classical control perspective, see for example Guo and Hern\'andez-Lerma \cite{Guo2001}. Similarly it may be possible to weaken the assumption of strictly balanced  EBSDE drivers to either balanced or weakly balanced.

The applications of the theory of EBSDEs are still in development, and the explicit computability of solutions to these equations, in terms of solving a single nonlinear vector equation, is of some interest. The consequences for risk sensitive ergodic control, and properties of the convergence of these solutions to the diffusion case, remain to be explored.

\bibliographystyle{plain}
\bibliography{../../RiskPapers/General}
\end{document}